\documentclass[a4paper, 11pt]{amsart}
\usepackage[latin1]{inputenc}
\usepackage{amssymb, amsthm, amsmath, thmtools, mathtools}
\usepackage{amsfonts}
\usepackage[abbrev]{amsrefs} 

\usepackage{graphicx}
\usepackage{hyperref}
\usepackage[bottom, marginal]{footmisc}
\usepackage{todonotes}

\declaretheoremstyle[headfont=\normalfont]{normalhead}
\newtheorem{lemma}{Lemma}[section]
\newtheorem{theorem}[lemma]{Theorem}
\newtheorem{proposition}[lemma]{Proposition}
\newtheorem{corollary}[lemma]{Corollary}
\newtheorem{definition}[lemma]{Definition}
\newtheorem{remark}[lemma]{Remark}

\newcounter{mt}

\newtheorem{maintheorem}[mt]{Theorem}

\newtheorem*{acknowledgement}{Acknowledgement}

\newcommand{\R}{\mathbb{R}}
\newcommand{\C}{\mathbb{C}}
\newcommand{\Q}{\mathbb{Q}}

\DeclareMathOperator{\sign}{sign}

\DeclareMathOperator{\vol}{vol}

\DeclareMathOperator{\diam}{diam}

\DeclareMathOperator{\GL}{GL}

\DeclareMathOperator{\SO}{\mathrm{SO}}

\newcommand{\calP}{\mathcal{P}}
\newcommand{\calU}{\mathcal{U}}
\newcommand{\conv}{\mathrm{conv}}

\numberwithin{equation}{section}

\author{Jonas Knoerr}
\title{Rigid motion invariant valuations on polytopes}
\date{}

\newcommand{\Addresses}{{
		\bigskip
		\footnotesize
		
		Jonas Knoerr, \textsc{Institute of Discrete Mathematics and Geometry, TU Wien, Wiedner Hauptstrasse 8-10, 1040 Wien, Austria}\par\nopagebreak
		\textit{E-mail address}: \texttt{jonas.knoerr@tuwien.ac.at}
		
		\medskip
	}}
	
\makeatletter
\def\blfootnote{\xdef\@thefnmark{}\@footnotetext}
\makeatother

\makeindex
\begin{document}
\maketitle
\begin{abstract}
	We show that any measurable, translation and $\SO(n)$-invariant valuation on polytopes in $\R^n$ is a linear combination of the intrinsic volumes, which extends Hadwiger's classical characterization of rigid motion invariant continuous valuations on convex bodies. This result is based on a novel regularity result for translation invariant, measurable, $1$-homogeneous, and simple valuations. As a consequence, we obtain a similar characterization of the Steiner point map on polytopes.
\end{abstract}
\blfootnote{2020 \emph{Mathematics Subject Classification}. 52B45, 52B11, 52A39.\\
	\emph{Key words and phrases}. measurable valuation, polytopes, intrinsic volumes.\\}

\section{Introduction}

Given a family of sets $\mathcal{S}$, a map $\varphi:\mathcal{S}\rightarrow (\mathcal{A},+)$ into an abelian semi-group is called a valuation if
\begin{align*}
	\varphi(K)+\varphi(L)=\varphi(K\cup L)+\varphi(K\cap L)
\end{align*}
for all $K,L\in\mathcal{S}$ such that $K\cup L, K\cap L\in\mathcal{S}$. This notion goes back to Dehn's solution of Hilbert's Third Problem on the possibility of an elementary definition of volume. Dehn showed this to be impossible by constructing a rigid motion invariant valuation on polytopes that vanishes on lower dimensional sets and which is not a multiple of the Lebesgue measure. In fact, his construction shows that the space of rigid motion invariant real-valued valuations on polytopes is infinite dimensional.

The situation is drastically different if additional regularity assumptions are imposed. The investigation of continuous valuations on the space $\mathcal{K}^n$ of convex bodies, that is, the set of all nonempty compact and convex subsets of $\R^n$ equipped with the Hausdorff metric, has been a very active area of mathematics. We refer to \cite{AleskerFaifmanConvexvaluationsinvariant2014,AleskerDescriptioncontinuousisometry1999,BernigEtAlHardLefschetztheorem2024,BoeroeczkyLudwigMinkowskivaluationslattice2019,FaifmanHofstaetterConvexvaluationsWhitney2025,FreyerEtAlUnimodularvaluationsEhrhart2025,LudwigEllipsoidsmatrixvalued2003,LudwigReitznerclassification$SLn$invariant2010,SchusterWannerer$GLn$contravariantMinkowski2012,SchusterWannererMinkowskivaluationsgeneralized2018,KnoerrSmoothvaluationsconvex2025,SchneiderSimplevaluationsconvex1996} for some of these developments for valuations on convex bodies and polytopes.

The most famous classical result is the following characterization of the intrinsic volumes $V_0,\dots,V_n$ due to Hadwiger.
\begin{theorem}[\cite{HadwigerVorlesungenuberInhalt1957}]\label{theorem:HadwigerConvexBodies}
	Let $\varphi:\mathcal{K}^n\rightarrow\R$ be a continuous, translation and $\SO(n)$-invariant valuation. Then there exist $c_0,\dots,c_n\in\R$ such that
	\begin{align*}
		\varphi=\sum_{k=0}^nc_k V_k.
	\end{align*}
\end{theorem}
This result implies in particular that the only valuation in this class that vanishes on lower dimensional convex bodies is the volume $V_n$.\\

In addition to Hadwiger's original proof (and the streamlined version given by Chen \cite{Chensimplifiedelementaryproof2004}), Klain obtained a second proof of this result using a more analytic approach based on generalized zonoids and properties of the cosine transform \cite{KlainshortproofHadwigers1995}. His proof provides a characterization of the volume as the unique continuous and translation invariant even valuation (i.e. $\varphi(-K)=\varphi(K)$) on convex bodies vanishing on lower dimensional sets, a characterization which is of fundamental importance in modern valuation theory and is one of the key ingredients in Alesker's proof of the so-called Irreducibility Theorem (see \cite{AleskerDescriptiontranslationinvariant2001}, and \cite{HofstaetterKnoerrLocalizationvaluationsAleskers2025} for an alternative approach). A key outcome of these developments is Alesker's discovery that Hadwiger-type finiteness results hold for all compact subgroups of $\GL(n,\R)$ that operate transitively on the unit sphere \cite{AleskerP.McMullensconjecture2000}. These groups are explicitly known, see \cite{MontgomerySamelsonTransformationgroupsspheres1943,BorelSomeremarksLie1949}, and there has been a substantial effort in characterizing the corresponding spaces of continuous valuations \cite{BernigHadwigertypetheorem2009,BernigInvariantvaluationsquaternionic2012,BernigIntegralgeometry$G_2$2011,BernigSolanesKinematicformulasquaternionic2017,AleskerHardLefschetztheorem2003,KotrbatyWannererIntegralgeometryoctonionic2025}, which has led to immense advances in integral geometry \cite{FuStructureunitaryvaluation2006,BernigFuHermitianintegralgeometry2011,BernigHugKinematicformulastensor2018,WannererIntegralgeometryunitary2014,Wannerermoduleunitarilyinvariant2014}.

In contrast, Hadwiger's proof relies almost entirely on decompositions of polytopes but requires that the valuation is defined on the larger space of all convex bodies at a crucial step. It has therefore been a longstanding question whether his characterization holds when restricted to the subspace $\mathcal{P}^n\subset \mathcal{K}^n$ of convex polytopes, compare \cite{McMullenSchneiderValuationsconvexbodies1983}*{Problem 15.4}. In this article, we provide an affirmative answer to this question, with an even weaker regularity assumption: We will only require that $\varphi:\calP^n\rightarrow\R$ is measurable, i.e. that the preimage of any open subset of $\R$ under $\varphi$ is a Borel set with respect to the topology induced by the Hausdorff metric. 
\begin{maintheorem}
	\label{maintheorem:HadwigerPolytopes}
	Let $\varphi:\calP^n\rightarrow\R$ be a measurable, translation and $\SO(n)$-invariant valuation. Then there exist $c_0,\dots,c_n\in\R$ such that
	\begin{align*}
		\varphi=\sum_{k=0}^nc_k V_k.
	\end{align*}
\end{maintheorem}
Let us remark that this case covers both continuous and monotone valuations on $\calP^n$, since a monotone and translation invariant valuation on polytopes is automatically continuous by a result due to McMullen \cite{McMullenValuationsEulertype1977}, and so in particular measurable. However, \cite{McMullenSchneiderValuationsconvexbodies1983}*{Problem 15.4} poses the same question for several other interesting cases that are not covered by the previous result. In particular, the corresponding question for locally bounded or non-negative valuations is still open. Let us also mention that for valuations on polytopes that are invariant under the special linear group, complete characterizations without any regularity assumptions have been established by Ludwig and Reitzner \cite{LudwigReitzner$SLn$invariantvaluations2017}.

As in the proofs of \autoref{theorem:HadwigerConvexBodies} given by Klain and Chen, we rely on the equivalence of the previous characterization to the following description of simple rigid motion invariant valuations (see \cite{KlainshortproofHadwigers1995,Chensimplifiedelementaryproof2004} for proofs of the corresponding statements for continuous valuations on $\mathcal{K}^n$). Here, a map $\varphi:\calP^n\rightarrow E$ into a real vector space $E$ is called simple if $\varphi(P)=0$ for all polytopes $P\in\calP^n$ that are contained in a proper affine subspace.
\begin{maintheorem}
	\label{maintheorem:VolumeCharacterization}
	Let $\varphi:\calP^n\rightarrow\R$ be a measurable, translation and $\SO(n)$-invariant, and simple valuation. Then there exists $c\in\R$ such that $\varphi=c V_n$.
\end{maintheorem}
Let us briefly discuss our approach and its relation to Hadwiger's original proof. As pointed out by McMullen and Schneider \cite{McMullenSchneiderValuationsconvexbodies1983}, Hadwiger's proof only uses that the valuations are defined on all convex bodies
in order to show that a $1$-homogeneous and simple valuation in this class has to vanish identically. This is achieved with the help of his characterization of the mean width (compare \cite{Chensimplifiedelementaryproof2004}*{Theorem~2.3}), which exploits the fact that a $1$-homogeneous translation invariant valuation is Minkowski additive in order to apply an averaging procedure to a convex body. For continuous and rigid motion invariant valuations on $\mathcal{K}^n$, this implies that a $1$-homogeneous valuation is already determined by its values on Euclidean balls. At this step, the argument cannot be generalized to continuous valuations defined solely on polytopes, since the limit of this averaging procedure (a ball with the same mean width) does not belong to the domain of the valuation.

The main idea underlying our approach is to instead average a suitable representation of a $1$-homogeneous translation invariant valuation on $\calP^n$, which is the reason why the argument extends to measurable instead of just continuous valuations. The key ingredient is a representation of so-called weakly continuous valuations on polytopes (see \autoref{section:valuationsPolytopes} for the definition) due to Hadwiger \cite{HadwigerTranslationsinvarianteadditiveund1952}, which provides a formula for such valuations in terms of certain odd functions on Stiefel manifolds. In Hadwiger's original proof of this representation formula, the relation between this function and the underlying valuation is hidden by an involved induction scheme, however, at least for $1$-homogeneous valuations, the proof can be adapted to show that this function may be chosen by evaluating the valuation in a family of orthogonal simplices, compare \autoref{theorem:1HomSimple}. In particular, it inherits certain regularity properties of the valuation, which are sufficient to show the required triviality result for $1$-homogeneous simple valuations, see \autoref{prop:simple1homHadwiger}.

This approach also works for translation invariant and $\SO(n)$-equivariant valuations with values in a general finite dimensional representation of $\SO(n)$, see \autoref{theorem:Equiv1homValuation}. As a further application, we extend Schneider's characterization of the Steiner point map \cite{SchneiderSteinerpointsconvex1971} from convex bodies to the polytopal setting. Here, the Steiner point of $K\in\mathcal{K}^n$ is given by
\begin{align*}
	s(K):=n\int_{S^{n-1}}vh_K(v)d\sigma(v),
\end{align*}
where $h_K(v):=\sup_{x\in K}\langle x,v\rangle$ and $\sigma$ denotes the spherical Lebesgue measure. Note that $s$ has the following properties:
\begin{enumerate}
	\item $s$ is translation and $\SO(n)$-equivariant: For all $P\in\calP^n$, $g\in \SO(n)$, and $x\in\R^n$,
		\begin{align*}
			s(gP+x)=gs(P)+x.
		\end{align*}
	\item $s$ is Minkowski additive:  $s(P_1+P_2)=s(P_1)+s(P_2)$ for $P_1,P_2\in\calP^n$.
\end{enumerate}
Schneider showed that any continuous map on $\mathcal{K}^n$ with these properties coincides with $s$. We establish the following version in the polytopal setting.
\begin{maintheorem}\label{maintheorem:SteinerPoint}
	Let $\varphi:\calP^n\rightarrow\R^n$ be measurable, translation and rotation equivariant, and Minkowski additive. Then $\varphi$ coincides with the Steiner point map.
\end{maintheorem}
Our proof of this result differs substantially from Schneider's proof for the corresponding result for continuous valuations on convex bodies. His proof exploits that a continuous Minkowski additive map on $\mathcal{K}^n$ extends to a $\SO(n)$-equivariant linear map on $C^2$-functions on the unit sphere (see also \cite{GoodeyWeilDistributionsvaluations1984} for a more general version of this construction), which reduces the proof to an application of Schur's Lemma. Our proof more closely resembles the approaches to  \autoref{theorem:HadwigerConvexBodies} through the description of simple valuations, however, we would like to point out that this description is in turn based on the properties of the $\SO(n)$-representations in the space of continuous functions on the relevant Stiefel manifold, which mirrors Schneider's approach.

Let us add the following remarks before we discuss the plan of the article. First, let us mention that several classification results for valuations on convex bodies can be reduced to Hadwiger's classification in \autoref{theorem:HadwigerConvexBodies} or Schneider's characterization of the Steiner point. If the remaining steps of these arguments can be formulated entirely within the class of polytopes, then \autoref{maintheorem:HadwigerPolytopes} may be used to obtain a  corresponding classification result for valuations on polytopes. This applies, for example, to the characterization of the intermediate area measures by Schneider \cite{SchneiderKinematischeBeruhrmaefur1975} and the moment vector by Hadwiger and Schneider \cite{HadwigerSchneiderVektorielleIntegralgeometrie1971}. We leave the details to the reader.\\
Second, in light of Alesker's finiteness result for continuous translation invariant valuations on $\mathcal{K}^n$ that are invariant under compact subgroups of $\GL(n,\R)$ operating transitively on the unit sphere, it is a natural question whether similar results hold in the polytopal setting of \autoref{maintheorem:HadwigerPolytopes}. For measurable valuations, the answer is no - there exist natural constructions of such valuations in terms of weighted sums over faces of suitable dimensions which provide ample examples (see, for example, \cite{Aleskerextendabilitycontinuityvaluations2014,BreidingEtAlzonoidalgebrageneralized2022,Wannererextendabilitycontinuityangular2020}). It would be interesting to know whether such finiteness results hold under stronger regularity assumptions on the valuations, for example local boundedness or continuity on the class of polytopes. Note that these examples cannot be locally uniformly continuous since they would otherwise extend to continuous valuations on $\mathcal{K}^n$. In particular, it seems to be nontrivial to construct an example that is continuous on $\calP^n$ but does not extend to $\mathcal{K}^n$.

\subsection{Plan of the article}
Since our approach heavily relies on Hadwiger's characterization of weakly continuous valuations, we first need to show that a measurable and translation invariant valuation is weakly continuous (which does not hold for more general valuations). This is based on results by McMullen \cite{McMullenWeaklycontinuousvaluations1983} that relate weak continuity to the polynomial behavior of translation invariant valuations on polytopes under Minkowski addition. We therefore include a short discussion of polynomials between abelian groups in \autoref{section:polynomiality}. This is in particular relevant, since a key part of an argument by McMullen seems to implicitly rely on a (to me) non-obvious equivalence between two a priori distinct notions of weak continuity for translation invariant valuations, compare the discussion in \autoref{section:weaklyContValuations} and in particular \autoref{theorem:equivJointSepCont}.

The remaining article is structured as follows: In addition to the discussion on notions of polynomiality, \autoref{section:background} includes the necessary background on convex polytopes including the Canonical Simplex Decomposition, as well as some basic facts on representations of $\SO(n)$, which are used in the proof of \autoref{maintheorem:SteinerPoint}.

\autoref{section:valuationsPolytopes} contains some well known facts about translation invariant valuations on polytopes and the relation of the homogeneous decomposition to weak continuity. In particular, this section is used to show that measurable translation invariant valuations are weakly continuous. Most of the results are simple adaptations of well known constructions, but we include the details for the convenience of the reader.

\autoref{section:CharacterizationSimpleWeak} contains the refined version of Hadwiger's characterization of simple weakly continuous valuations of degree $1$. The key result is \autoref{theorem:Equiv1homValuation}, which is the central technical contribution of this article. The main results are obtained from this result in \autoref{section:HadwigerPolytope} and \autoref{section:SteinerPoint}, respectively.
\begin{acknowledgement}
	This project originates in a discussion on possible extensions of Hadwiger's proof of \autoref{theorem:HadwigerConvexBodies} to smaller families of convex bodies at the Erd\H{o}s Center at the Alfr\'ed R\'enyi Institute of Mathematics during the "Focused Workshop on Current Trends in Geometric Valuation Theory" in 2025, and I want to thank the participants and organizers for the discussions and the Erd\H{o}s Center for the hospitality.
	\\
	This research was funded in whole or in part by the Austrian Science Fund (FWF), \href{https://www.doi.org/10.55776/PAT4205224}{10.55776/PAT4205224}. 
\end{acknowledgement} 

\section{Background}
	\label{section:background}
	\subsection{Notions of polynomiality}
		\label{section:polynomiality}
		We will rely on the following notion of polynomiality, which is also used in \cite{PukhlikovKhovanskiiFinitelyadditivemeasures1992}.		
		\begin{definition}\label{def:polynomial}
			Given two abelian groups $N,L$, a map $h:N\rightarrow L$ is called a polynomial of degree at most $m\in\mathbb{N}$ if the following two conditions hold:
			\begin{enumerate}
				\item if $m=0$, then $h$ is constant,
				\item if $m>0$, then for any $a\in N$, the map $h_a:N\rightarrow L$, $h_a(x)=h(x+a)-h(x)$, is a polynomial of degree at most $m-1$.
			\end{enumerate}
		\end{definition}
		Equivalently, for all $x,a_1,\dots,a_{m+1}\in N$,
		\begin{align}
			\label{eq:polynomialabelianGroups}
			\sum_{k=0}^{m+1}(-1)^{k} \sum_{1\le i_1<\dots <i_k\le m+1}h(x+a_{i_1}+\dots+a_{i_k})=0
		\end{align}
		The following results follow easily using the inverse of the Vandermonde matrix as well as Eq.~\eqref{eq:polynomialabelianGroups}.
		\begin{corollary}
			If $N$ is an abelian group and $L$ is a $\Q$ vector space, then any polynomial $h:N\rightarrow L$ of degree at most $m$ is a sum of homogeneous polynomials of degree at most $m$: There exist unique polynomials $h_j:N\rightarrow L$ of degree at most $j$ such that
			\begin{align*}
				h(x)=\sum_{j=0}^m h_j(x)\quad \text{for}~x\in N,
			\end{align*}
			where $h_j(nx)=n^jh_j(x)$ for all $n\in\mathbb{Z}$.
		\end{corollary}
	\begin{corollary}\label{corollary:homogeneousDecompPolynomials}
		If $N,L$ are vector spaces over $\mathbb{Q}$ and $h:N\rightarrow L$ is a polynomial of degree at most $m\in\mathbb{N}$, then there exist unique symmetric maps $\lambda_j\in \mathrm{Hom}_\Q(N^{\otimes j},L)$, $0\le j\le m$ such that
		\begin{align*}
			h\left(\sum_{i=1}^n a_i\right)=\sum_{j=0}^m \sum_{1\le i_1,\dots,i_j\le n } \lambda_j(a_{i_1}\otimes\dots\otimes a_{i_j})
		\end{align*}
		for all $a_j\in N$, $n\in\mathbb{N}$.
	\end{corollary}
		If $h:N\rightarrow L$ is a polynomial, we call $h$ an $\R$-polynomial if the maps $\lambda_j\in \mathrm{Hom}_\Q(N^{\otimes j},L)$ from \autoref{corollary:homogeneousDecompPolynomials} are $\R$-linear. In order to avoid ambiguities, we will call a polynomial in the sense of \autoref{def:polynomial} defined between $\Q$-vector spaces a $\Q$-polynomial. In the homogeneous case, we will therefore also distinguish between real homogeneous or rational homogeneous polynomials. For polynomials on (finite dimensional) real vector spaces, $\R$-polynomiality is equivalent to several regularity properties of the underlying function, compare \cite{Kuczmaintroductiontheoryfunctional2009}.
		\begin{lemma}
			\label{lemma:finiteDimPolynomials}
			Let $f:\R^n\rightarrow \R^m$ be a $\Q$-polynomial. The following are equivalent:
			\begin{enumerate}
				\item $f$ is an $\R$-polynomial.
				\item $f$ is continuous on a nonempty open set.
				\item $f$ is separately continuous on a nonempty open set.
				\item There exists a nonempty open set $U\subset\R^n$ such that $f|_U$ is Borel measurable.
				\item There exists a nonempty open set $U\subset\R^n$ such that $f|_U$ is bounded. 
			\end{enumerate}
		\end{lemma}

	\subsection{Polytopes}
		For a subset $A\subset\R^n$, we denote by $\conv(A)$ its convex hull. By definition, a polytope $P$ is the convex hull of a finite subset in $\R^n$ and thus in particular a convex body. We consider the space $\calP^n$ of all convex polytopes as a subspace of the space of  $\mathcal{K}^n$. In particular, it is a metric space with respect to the Hausdorff metric. In order to simplify the notation, we will sometimes write $[x_1,\dots,x_k]:=\conv(\{x_1,\dots,x_k\})$ for $x_1,\dots,x_k\in \R^n$. In particular, $[x_1,x_2]$ is a line segment if $x_1\ne x_2$. If $E\subset\R^n$ is an affine subspace, we denote by $\calP(E)\subset \calP^n$ the subspace of polytopes contained in $E$.
		
		Since we will be interested in simple valuations, we will use the following notation: If $P, P_1,\dots P_k\in\calP^n$ are polytopes such that $P=\bigcup_{i=1}^k P_i$ and such that $P_i\cap P_j$ is a lower dimensional polytope (or the empty set) for all $1\le i<j\le k$, then we write
		\begin{align}
			\label{eq:polytopeTriangulation}
			P=P_1\sqcup\dots\sqcup P_k=\bigsqcup_{i=1}^k P_i.
		\end{align}
		We will  use this mostly for the following decomposition of simplices. By definition, for $0\le k\le n$, a $k$-dimensional simplex $S$ is the convex hull of $k+1$ affinely independent points $p_0,\dots,p_{k}$. We set $x_0=p_0$ and $x_i:=p_i-p_{i-1}$ for $1\le i\le k$. Then it is easy to see that
		\begin{align*}
			S=\left\{ x_0+\sum_{i=1}^k r_i x_i : 1\ge r_1\ge \dots\ge r_k\ge 0\right\},
		\end{align*}
		see for example \cite{LudwigMussnigValuationsConvexBodies2023}*{Lemma~4.1}. Following \cite{LudwigMussnigValuationsConvexBodies2023}, we write $S=\langle x_0;x_1,\dots,x_k\rangle$ in this case. The following result is known as the Canonical Simplex Decomposition \cite{HadwigerVorlesungenuberInhalt1957}*{Section 1.2.6}. We refer to \cite{AleskerIntroductiontheoryvaluations2018}*{Theorem~1.1.} and \cite{LudwigMussnigValuationsConvexBodies2023}*{Theorem~4.2} for alternative proofs.
		\begin{theorem}
			\label{theorem:CanonicalSimplexDecomp}
			Let $S=\langle x_0;x_1,\dots,x_n\rangle$ be an $n$-dimensional simplex. Define for $0\le k\le n$
			\begin{align*}
				\underline{S}_k:=\langle x_0;x_1,\dots,x_k\rangle, \quad\text{and}\quad 	\overline{S}_{n-k}:=\left\langle x_0+\sum_{i=1}^kx_i;x_{k+1},\dots,x_n\right\rangle.
			\end{align*}
			Then
			\begin{align*}
				S=\bigsqcup_{k=0}^n \left((1-t)\underline{S}_k+t\overline{S}_{n-k}\right)
			\end{align*}
			for $0<t<1$.
		\end{theorem}
		\begin{remark}\label{remark:simplexDecompOrthogonal}
			If $S=\langle x_0;x_1,\dots,x_n\rangle$ is an orthogonal simplex, i.e. if $x_1,\dots,x_n$ are mutually orthogonal, then $\underline{S}_k$ and $\overline{S}_{n-k}$ belong to orthogonal affine subspaces for $1\le k\le n-1$. 
		\end{remark}
		Note further that any $n$-dimensional polytope admits a decomposition of the form in Eq.~\eqref{eq:polytopeTriangulation} with simplices $P_i$ whose pairwise intersection is again a lower dimensional simplex, i.e. every polytope admits a triangulation.
	
	\subsection{Representation theory of $\SO(n)$}
		The proof of \autoref{maintheorem:SteinerPoint} relies on some simple representations-theoretic results for the compact Lie group $\SO(n)$. Since the results hold for arbitrary compact Lie groups, we state them in greater generality. For an introduction to representations of compact Lie groups, we refer to \cite{SepanskiCompactLiegroups2007}.
		
		Let $(E,\pi)$ be a continuous representation of a compact Lie group $G$ on a topological vector space $E$ over $\C$, i.e. such that the map
		\begin{align*}
			G\times E&\rightarrow E\\
			(g,v)&\mapsto \pi(g)v
		\end{align*}
		is continuous. Recall that any continuous and irreducible representation of a compact Lie group is finite dimensional. We denote by  $\widehat{G}$ the set of all equivalence classes of isomorphic irreducible representations, and for an irreducible representation $(E_\pi,\pi)$, we denote by $[\pi]$ the equivalence class of all irreducible representations of $G$ isomorphic to $E_\pi$. 
		
		A vector $v\in E$ is called a $G$-finite vector if 
		\begin{align*}
			\mathrm{span}\{\pi(g)v:g\in G\}
		\end{align*}
		is a finite dimensional subspace. We will denote the space of $G$-finite vectors by $E_{G-\mathrm{fin}}$. Then the $G$-finite vectors in $C(G)$ with respect to the left or right action of $G$ coincide, compare \cite{SepanskiCompactLiegroups2007}*{Theorem 3.21}. These actions are defined respectively by
		\begin{align*}
			&L_gf (h)=f(g^{-1}h), &&R_gf (h)=f(hg), \quad g,h\in G
		\end{align*}
		for $f\in C(G)$. The following result is called the Canonical Decomposition of $C(G)_{G-\mathrm{fin}}$.
		\begin{theorem}\label{thm:CanonicalDecomposition}
			Let $(g_1,g_2)\in G\times G$ act on $C(G)$ by $R_{g_1}\circ L_{g_2}$. Then we have a $G\times G$-equivariant decomposition
			\begin{align*}
				C(G)_{G-\mathrm{fin}}\cong \bigoplus_{[\pi]\in \widehat{G}} E_\pi^*\otimes E_\pi,
			\end{align*}
			where the last isomorphism is induced by
			\begin{align*}
				E_\pi^*\otimes E_\pi \ni \phi\otimes x\mapsto \left[g\mapsto \phi(\pi(g^{-1})x)\right]\in C(G).
			\end{align*}
		\end{theorem}
		In the cases we will be interested in, we will only have the action of the second factor, i.e. we will consider $C(G)$ as a representation of $G$ with respect to left multiplication. More precisely, we will consider $G=\SO(n)$, a finite dimensional representation $(E,\pi_E)$ of $\SO(n)$ and the spaces $C(\SO(n),E)\cong C(\SO(n))\otimes E$. Consider the spaces $C(\SO(n),E)^{\SO(n)}$ of $\SO(n)$-invariant elements, i.e. all continuous functions $f:\SO(n)\rightarrow E$ satisfying $\pi_E(h)f(h^{-1}g)=f(g)$ for $h,g\in \SO(n)$. Note that this is equivalent to $f(hg)=\pi_E(h)f(g)$. The following is a simple consequence, however, we include the argument for the convenience of the reader.
		\begin{corollary}\label{corollary:canonicalDecompTensor}
			If $E$ is a finite dimensional representation of $\SO(n)$, then  
			\begin{align*}
				C(\SO(n),E)^{\SO(n)}\cong \bigoplus_{[\pi]\in\widehat{\SO(n)}}E_\pi^*\otimes (E_\pi\otimes E)^{\SO(n)},
			\end{align*}
			where the sum contains only finitely many nontrivial terms.
		\end{corollary}
		\begin{proof}
			Since $E$ is finite dimensional, the $\SO(n)$-finite vectors in $C(\SO(n), E)\cong C(\SO(n))\otimes E$ coincide with $C(\SO(n))^{\SO(n)-\mathrm{fin}}\otimes E$. Using the decomposition in \autoref{thm:CanonicalDecomposition}, we obtain an $\SO(n)$-equivariant isomorphism
			\begin{align*}
				C(\SO(n),E)_{\SO(n)-\mathrm{fin}}\cong \bigoplus_{[\pi]\in\widehat{\SO(n)}}E_\pi^*\otimes (E_\pi\otimes E),
			\end{align*}
			where $\SO(n)$ acts trivially on the first factor on the right. Since any $\SO(n)$-invariant element in $C(\SO(n),E)$ is $\SO(n)$-finite by definition, this provides the desired isomorphism by passing to the $\SO(n)$-invariant elements. For the last claim, note that
			\begin{align*}
				(E_\pi\otimes E)^{\SO(n)}\cong\mathrm{Hom}_{\SO(n)}(E^{*}_\pi,E)
			\end{align*}
			is isomorphic to the space of $\SO(n)$-equivariant homomorphisms from $E^*_\pi$ to $E$. Since $E$ is finite dimensional, Schur's Lemma implies that this space is nontrivial for only a finite number of representations $[\pi]\in \widehat{\SO(n)}$. 	
		\end{proof}
		
\section{Valuations on polytopes}
	\label{section:valuationsPolytopes}
	
	As a general background on valuations on polytopes and convex bodies, we refer to \cite{SchneiderConvexbodiesBrunn2014}*{Section 6}. The results of the next subsection are minor modifications of well-known results, which we include with proofs in order to translate the results to the specific case of measurable valuations.
	
	Several of our arguments rely on the following implication of the valuation property for polytopes, see	\cite{SchneiderConvexbodiesBrunn2014}*{Theorem~6.2.3}.
	\begin{theorem}\label{theorem:InclusionExclusionPrinciple}
		Every valuation $\varphi:\calP^n\rightarrow E$ satisfies the Inclusion-Exclusion Principle: If $P_i\in \calP^n$, $1\le i\le N$, are polytopes such that $\bigcup_{i=1}^N P_i\in \calP^n$, then
		\begin{align*}
			\varphi\left(\bigcup_{i=1}^NP_i\right)=\sum_{r=1}^N(-1)^{r-1}\sum_{1\le i_1< \dots< i_r\le N}\varphi\left(\bigcap_{j=1}^rP_{i_j}\right)
		\end{align*}
	\end{theorem}
	Note that the corresponding result does not hold for general valuations on convex bodies, however, it does hold for continuous valuations, see \cite{KlainRotaIntroductiongeometricprobability1997}*{Theorem~5.1.1}. The Inclusion-Exclusion Principle admits the following more algebraic interpretation: Consider the space $Z(\R^n)$ of all functions $\alpha:\R^n\rightarrow \mathbb{Z}$ that are $\mathbb{Z}$-linear combinations of indicator functions of polytopes in $\calP^n$. Groemer's Extension Theorem (see \cite{Groemerextensionadditivefunctionals1978} or \cite{KlainRotaIntroductiongeometricprobability1997}*{Theorem 2.2.1}) implies the following
	\begin{corollary}
		\label{corollary:Groemer}
		Let $E$ be an abelian group. Every valuation $\varphi:\calP^n\rightarrow E$ induces a unique well-defined $\mathbb{Z}$-linear map $\tilde{\varphi}:Z(\R^n)\rightarrow E$ satisfying $\tilde{\varphi}(1_P)=\varphi(P)$ for $P\in\mathcal{P}^n$.
	\end{corollary}
	
	\subsection{Polynomiality of translation invariant valuations}
		For $P,Q\in\calP^n$, let us denote their Minkowski sum by
		\begin{align*}
			P+Q=\{x+y:x\in P,y\in Q\}.
		\end{align*}
		This defines a continuous binary operation on $\calP^n$ and equips the space of polytopes with the structure of a semi-group with cancellation law. Moreover, translation invariant valuations satisfy the following remarkable polynomiality behavior with respect to Minkowski addition, as shown by McMullen.
		\begin{theorem}[\cite{McMullenValuationsEulertype1977}*{Theorem~6}]\label{theorem:PolynomialityMcMullen}
		Let $E$ be a $\Q$-vector space, and let $\varphi:\calP^n\rightarrow E$ be a translation invariant valuation. Then for any $P_1,\dots,P_N\in\calP^n$ and rational $\lambda_1,\dots,\lambda_N\ge 0$, the map
		\begin{align*}
			(\lambda_1,\dots,\lambda_N)\mapsto \varphi\left(\sum_{j=1}^N\lambda_jP_j\right)
		\end{align*}
		is a polynomial of degree at most $n$. Moreover, as a function of $P_1,\dots,P_k$, the coefficient of $\lambda_1^{r_1}\dots\lambda_k^{r_k}$ is a translation invariant valuation and homogeneous of degree $r_i$  in each argument.
	\end{theorem}	
	Let us remark that this result holds in greater generality for so-called polynomial valuations, as shown by Khovanski{\u i} and Pukhlikov \cite{PukhlikovKhovanskiiFinitelyadditivemeasures1992}. In fact, their results show that any translation invariant valuation on $\calP^n$ induces a polynomial of degree at most $n$ in the sense of \autoref{def:polynomial} on the group of virtual polyhedra, which is naturally isomorphic to the Grothendieck group of $(\calP^n,+)$, see \cite{PukhlikovKhovanskiiFinitelyadditivemeasures1992}*{\S 6 Corollary 2}, and which they realize as a multiplicative subset of the ring $Z(\R^n)$ equipped with a multiplication induced by Minkowski addition. This interpretation will be relevant in the next section.
	
	For now, we discuss some well-known implications of \autoref{theorem:PolynomialityMcMullen}. Since the compatibility of these results with measurability is crucial for the further constructions, we include the necessary arguments.	
	\begin{corollary}
		Let $E$ be a $\Q$-vector space and $\varphi:\calP^n\rightarrow E$ a translation invariant valuation. Then there exist unique translation invariant valuations $\varphi_j:\calP^n\rightarrow E$ rational homogeneous of degree $j$ for $0\le j\le n$ such that $\varphi=\sum_{k=0}^n\varphi_k$. More precisely, there exist constants $c_{ij}\in\Q$, $0\le i,j\le n$, independent of $\varphi$ and $P\in\calP^n$ such that
		\begin{align}
			\label{eq:homComponentVandermonde}
			\varphi_i(P)=\sum_{j=0}^n c_{ij}\varphi(jP).
		\end{align}
	\end{corollary}
	\begin{proof}
		The numbers $\varphi_j(P)$ are obviously the coefficients of the polynomial $t\mapsto \varphi(tP)=\sum_{j=0}^nt^j\varphi_j(P)$ for rational $t\ge0$. If we plug in $t=0,\dots,n$, we obtain a system of $n+1$ linear equations 
		\begin{align*}
			\begin{pmatrix}
				1 & 0 &\dots & 0\\
				1 & 1 &\dots & n\\
				\vdots &\vdots &  &\vdots\\
				1 &n&\dots & n^n
			\end{pmatrix}\begin{pmatrix}
				\varphi_0(P)\\
				\varphi_1(P)\\
				\vdots\\
				\varphi_n(P)
			\end{pmatrix}=\begin{pmatrix}
				\varphi(0\cdot P)\\
				\varphi(1\cdot P)\\
				\vdots\\
				\varphi(n\cdot P)
			\end{pmatrix},
		\end{align*}
		where the matrix on the left hand side is a Vandermonde matrix with rational entries. We may thus invert this matrix in order to obtain the desired coefficients $c_{ij}\in \Q$, which therefore do not depend on $P\in\calP^n$ and the valuation $\varphi$.
		
		From the representation of $\varphi_i$ in \eqref{eq:homComponentVandermonde} it is clear that $\varphi_i$ is a translation invariant valuation. Moreover, since 
		\begin{align*}
			\sum_{j=0}^ns^j\varphi_j(tP)=\varphi(stP)=\sum_{j=0}^ns^jt^j\varphi_j(P)
		\end{align*}
		for all rational $s,t\ge0$ and $P\in\calP^n$, comparing coefficients shows that $\varphi_j(tP)=t^j\varphi_j(P)$ for rational $t\ge0$. Thus $\varphi_j$ is rational $j$-homogeneous.
	\end{proof}
	
	\begin{corollary}\label{corollary:PropertiesHomComp}
		Let $E$ be a topological vector space over $\Q$, $\varphi:\calP^n\rightarrow E$ a translation invariant valuation and $\varphi=\sum_{j=0}^n\varphi_j$ its homogeneous decomposition. 
		\begin{itemize}
			\item If $\varphi$ is locally bounded, then $\varphi_j$ is locally bounded for every $0\le j\le n$.
			\item If $\varphi$ is measurable, then $\varphi_j$ is measurable for every $0\le j\le n$.
		\end{itemize}
	\end{corollary}
	\begin{proof}
		This follows directly from the representation of $\varphi_j$ in terms of $\varphi$ in Eq.~\eqref{eq:homComponentVandermonde}, since the map $\calP^n\rightarrow\calP^n$, $P\mapsto \lambda P$ is continuous for every $\lambda\in \Q$ and maps bounded sets of $\calP^n$ to bounded sets.
	\end{proof}

	Note that if $\varphi:\calP^n\rightarrow\R$ is a translation invariant and rational $k$-homogeneous valuation, then for all $P_1,\dots,P_N\in\calP^n$ and rational $\lambda_1,\dots,\lambda_N\ge0$, the map
	\begin{align*}
		(\lambda_1,\dots,\lambda_N)\mapsto \varphi\left(\sum_{j=1}^N\lambda_j P_j\right)
	\end{align*}
	is a homogeneous polynomial of degree $k$.
	
	\begin{corollary}
		\label{corollary:MinkwskiAdditive1Hom}
		Let $E$ be a $\Q$-vector space and $\varphi:\calP^n\rightarrow E$ a translation invariant valuation that is rational homogeneous of degree $1$. Then $\varphi$ is Minkowski additive.
	\end{corollary}
	\begin{proof}
		Consider the map $(s,t)\mapsto \varphi(sP+tQ)-s\varphi(P)-t\varphi(Q)$ for two polytopes $P,Q\in\calP^n$ and rational $s,t\ge0$. Since $\varphi$ is rational $1$-homogeneous, this is a polynomial of degree $1$, so there exist $a,b\in E$ such that 
		\begin{align*}
			\varphi(sP+tQ)-s\varphi(P)-t\varphi(Q)=sa+tb
		\end{align*}
		for all rational $s,t\ge0$. If we set $s=0$ and $t=1$, we obtain $b=0$, while we obtain $a=0$ if we set $s=1$ and $t=0$. Thus 
		\begin{align*}
			\varphi(sP+tQ)=s\varphi(P)+t\varphi(Q),
		\end{align*}
		which for $s=t=1$ implies the result.
	\end{proof}
	
	The following is well known and essentially contained in \cite{McMullenValuationsEulertype1977}*{Section~3}.
	\begin{theorem}\label{theorem:mixValuation}
		Let $E$ be a $\Q$-vector space $\varphi:\calP^n\rightarrow E$ a translation invariant valuation. If $\varphi$ is rational $k$-homogeneous, then there exists a unique map $\bar{\varphi}:(\calP^n)^k\rightarrow E$ with the following properties:
		\begin{enumerate}
			\item $\bar{\varphi}$ is Minkowski additive in each argument.
			\item $\bar{\varphi}$ is symmetric.
			\item $\bar{\varphi}(P,\dots,P)=\varphi(P)$ for every $P\in\calP^n$.
		\end{enumerate}
		Moreover, there exist constants $c_{i_1,\dots,i_k}\in\Q$, $0\le i_1,\dots,i_k\le k$, such that
		\begin{align}
			\label{eq:constantsPolarization}
			\bar{\varphi}(P_1,\dots,P_k)=\sum_{i_1,\dots,i_k=0}^kc_{i_1,\dots,i_k}\varphi\left(\sum_{j=1}^k i_j P_j\right).
		\end{align}
	\end{theorem}
	\begin{proof}
		Let $\bar{\varphi}(P_1,\dots,P_k)$ be the coefficient of $\lambda_1\cdot\dots\cdot\lambda_k$ of the homogeneous polynomial 
		\begin{align*}
			(\lambda_1,\dots,\lambda_k)\mapsto\frac{1}{k!}\varphi\left(\sum_{j=1}^k\lambda_jP_j\right)
		\end{align*}
		for rational $\lambda_1,\dots,\lambda_k\ge 0$, compare \autoref{theorem:PolynomialityMcMullen}. We may apply the inverse of the Vandermonde matrix in each argument to obtain the desired constants $c_{i_1,\dots,i_k}\in\Q$, $0\le i_1,\dots,i_k\le k$, independent of $P$ and $\varphi$ such that Eq.~\eqref{eq:constantsPolarization} holds. Then $\bar{\varphi}$ satisfies the three properties listed above. Note in particular that $\bar{\varphi}$ is rational $1$-homogeneous since it is the coefficient of $\lambda_1\dots\lambda_k$, so it is Minkowski additive by \autoref{corollary:MinkwskiAdditive1Hom}.
	\end{proof}
	We will call the map $\bar{\varphi}$ the polarization of $\varphi$. It inherits the following properties from $\varphi$.	
	\begin{corollary}
		\label{corollary:PropertiesHomComponents}
		Let $E$ be a finite dimensional real vector space and $\varphi:\calP^n\rightarrow E$ a translation invariant valuation that is rational homogeneous of degree $k$.
		\begin{enumerate}
			\item If $\varphi$ is locally bounded, then $\bar{\varphi}$ is locally bounded on $(\calP^n)^k$.
			\item If $\varphi$ is measurable, then $\bar{\varphi}$ is measurable on $(\calP^n)^k$.
			\item If $\varphi$ is continuous, then $\bar{\varphi}$ is continuous on $(\calP^n)^k$.
		\end{enumerate}
	\end{corollary}
	\begin{proof}
		For all rational $\lambda_1,\dots,\lambda_k$, the map $(\calP^n)^k\rightarrow\calP^n$, $(P_1,\dots,P_k)\mapsto \sum_{j=1}^k \lambda_j P_j$ is continuous and maps bounded subsets to bounded subsets. From Eq.~\eqref{eq:constantsPolarization}, we see that $\bar{\varphi}$ is a linear combination of the composition of $\varphi$ with maps of this type, which implies the claim.
	\end{proof}
	
	If $\varphi:\calP^n\rightarrow E$ is a translation invariant valuation that is rational homogeneous of degree $k$, we denote by $(P,Q)\mapsto\bar{\varphi}(P[j],Q[k-j])$ the functional obtained by taking $P$ for $j$ arguments and $Q$ for $k-j$ arguments of the polarization $\bar{\varphi}$. Then \autoref{theorem:PolynomialityMcMullen} (or the representation of $\bar{\varphi}$ in \autoref{theorem:mixValuation}) shows that this is a translation invariant valuation in each argument that is rational homogeneous of degree $j$ and $k-j$ respectively. We will use the following observation in \autoref{section:HadwigerPolytope}.
	\begin{remark}
		\label{remark:propertiesPolarization}
		If we have a direct sum decomposition $\R^n=E\oplus F$ of real vector spaces and a translation invariant valuation $\varphi:\calP^n\rightarrow \R$ that is rational homogeneous of degree $k$, then we may define a map
		\begin{align*}
			\tilde{\varphi}:\calP(E)\times\calP(F)&\rightarrow\R\\
			(P,Q)&\mapsto \bar{\varphi}(P[j],Q[k-j])
		\end{align*}
		to obtain a functional that is a $j$-homogeneous valuation in $\calP(E)$ and $(k-j)$-homogeneous valuation on $\calP(F)$ and translation invariant in both arguments. Moreover, since $\bar{\varphi}$ is given by evaluating $\varphi$ in suitable sums of polytopes by \autoref{theorem:mixValuation}, we directly obtain the following additional properties:
		\begin{enumerate}
			\item If $\varphi$ is simple, then so is $\tilde{\varphi}$ in both arguments.
			\item If the direct sum $E\oplus F$ is orthogonal and $\varphi$ is $\SO(n)$-invariant, then $\tilde{\varphi}$ is $\SO(E)$-invariant in the first and $\SO(F)$-invariant in the second argument.
		\end{enumerate}
	\end{remark}
	\subsection{Weakly continuous valuations}
		\label{section:weaklyContValuations}
			For a finite tuple $(v_1,\dots,v_N)$ of unit vectors, let $\calP^n_{(v_1,\dots,v_N)}\subset \calP^n$ denote the subset of all polytopes $P\in\calP^n$ such that there are $\eta_1,\dots,\eta_N\in\R$ with
		\begin{align*}
			P=\{x\in \R^n: \langle x,v_j\rangle\le \eta_j, 1\le j\le N\}.
		\end{align*}
		We call a tuple $(v_1,\dots,v_N)$ of unit vectors \emph{admissible} if $P(\eta):=\{x\in \R^n: \langle x,v_j\rangle\le \eta_j, 1\le j\le N\}$ is compact for all $\eta=(\eta_1,\dots,\eta_N)\in\R^N$. In this case, this set is either empty or belongs to $\calP^n_{(v_1,\dots,v_N)}$. If $(v_1,\dots,v_N)$ is an admissible tuple, we consider the nonempty convex set \begin{align*}
			U_{(v_1,\dots,v_N)}:=\{(\eta_1,\dots,\eta_N)\in \R^N: P(\eta_1,\dots,\eta_N)\neq\emptyset\}.
		\end{align*}
		If $\varphi:\calP^n\rightarrow E$ is a valuation with values in a topological vector space, we may consider the restriction of $\varphi$ to $\calP^n_{(v_1,\dots,v_N)}$ as a function on $U_{(v_1,\dots,v_N)}$ by considering the function
		\begin{align*}
			(\eta_1,\dots,\eta_N)\mapsto \varphi(P(\eta_1,\dots,\eta_N))=\varphi\left(\{x\in \R^n: \langle x,v_j\rangle\le \eta_j, 1\le j\le N\}\right)
		\end{align*}
		Then $\varphi$ is called \emph{weakly continuous} if this function is \emph{jointly} continuous.
		\begin{remark}
			It is sometimes not clear in the literature whether weak continuity of valuations on polytopes is assumed to mean joint or only separate continuity of this map. It turns out that these two notions are in fact equivalent for translation invariant valuations, however, this result seems to not be explicitly available in the literature.
		\end{remark}
		
		\begin{theorem}
			\label{theorem:equivJointSepCont}
			The following are equivalent for a translation invariant valuation $\varphi:\calP^n\rightarrow E$ into a Hausdorff topological real vector space:
			\begin{enumerate}
				\item $\varphi$ is weakly continuous.
				\item  For each admissible tuple $(v_1,\dots,v_N)$ of unit vectors, the function 
				\begin{align}
					\label{eq:mapWeakContinuity}
					\begin{split}
						U_{(v_1,\dots,v_N)}&\rightarrow E\\
						(\eta)&\mapsto \varphi(P(\eta))	
					\end{split}					
				\end{align}
				is separately continuous.
			\end{enumerate}
		\end{theorem}
		\begin{proof}
			Note that it is sufficient to show that the map in Eq.~\eqref{eq:mapWeakContinuity} is continuous at every point $\eta\in U_{(v_1,\dots,v_N)}$ such that $P(\eta)$ is a full dimensional polytope, since we may otherwise use the inverse of the Vandermonde matrix to express $\varphi$ in terms of $\varphi(\cdot+tP(\eta_0))$, $t>0$ rational, for a full dimensional polytope $P(\eta_0)$. Thus, let $\eta_0$ be such that $P(\eta_0)$ is full dimensional. Let $J_0\subset \{1,\dots,N\}$ be the complement of the set of all $j$ such that $h_{P(\eta_0)}(v_j)<\eta_{0,j}$. By continuity, we find a relatively compact open and convex neighborhood $V\subset U_{(v_1,\dots,v_N)}$ such that for all $\eta\in V$ we have
			\begin{align*}
				h_{P(\eta)}(v_j)<\eta_j \quad \text{for all}~j\notin J_0.
			\end{align*}
			For $I\subset J_0$ let $V^I$ denote the set of all $\eta\in \overline{V}$ such that
			\begin{align*}
				&h_{P(\eta)}(v_i)<\eta_i \quad \text{for all}~i\notin I\\
				&h_{P(\eta)}(v_i)=\eta_i \quad \text{for all}~i\in I.
			\end{align*}			
			Then, by definition, we have for all $\eta\in \overline{V^I}$
			\begin{align*}
				P(\eta)=\left\{x: \langle x,v_i\rangle \le\eta_i~\text{for all}~i\in I\right\},
			\end{align*}
			where all inequalities are attained on $P(\eta)$. In particular, for $\eta_1,\eta_2\in \overline{V^I}$ and $\lambda\in [0,1]$, we have $\lambda \eta_1+(1-\lambda)\eta_2\in \overline{V^I}$ with
			\begin{align}
				\label{eq:localPolynomialityWeaklyCont}
				\lambda P(\eta_1)+(1-\lambda) P(\eta_2)=P(\lambda \eta_1+(1-\lambda)\eta_2).
			\end{align}
			This is seen as follows: Since the relevant inequalities are sharp, a point $x\in\R^n$ is contained in $\lambda P(\eta_1)+(1-\lambda) P(\eta_2)$ if and only if for all $i\in I$
			\begin{align*}
				\langle x,v_i\rangle \le& h_{\lambda P(\eta_1)+(1-\lambda) P(\eta_2)}(v_i)=\lambda h_{P(\eta_1)}(v_i)+(1-\lambda) h_{P(\eta_2)}(v_i)\\
				=&\lambda\eta_{1,i}+(1-\lambda)\eta_{2,i}.
			\end{align*}
			This is the case if and only if $x\in P(\lambda \eta_1+(1-\lambda)\eta_2)$.\\
			
			Now let us restrict the map in Eq.~\eqref{eq:mapWeakContinuity} to the closed convex neighborhood $\overline{V}$ of $\eta_0$. By construction, $\overline{V}$ is covered by the finite family of closed convex sets $\overline{V^I}$. In particular, $\overline{V}$ is covered by the finite subfamily consisting of full dimensional convex sets. The restriction of the map in Eq.~\eqref{eq:mapWeakContinuity} to the interior of one of the full dimensional sets $\overline{V^I}$ is then separately continuous and the restriction of a $\Q$-polynomial due to Eq.~\eqref{eq:localPolynomialityWeaklyCont} (and \autoref{theorem:PolynomialityMcMullen}). Thus this is the restriction of an $\R$-polynomial, which is therefore continuous on $\overline{V^I}$. As $\overline{V}$ is covered by the finite union of these closed sets, this shows that the restriction of this map to $\overline{V}$ is continuous.
		\end{proof}
	
		We rely on the following result relating weak continuity and $\R$-polynomiality for translation invariant valuations, which is due to McMullen.	
		\begin{theorem}[\cite{McMullenValuationsEulertype1977}*{Theorem~9}]\label{theorem:RealHomogeneity}
			The following are equivalent for a translation invariant valuation $\varphi:\calP^n\rightarrow E$ into a Hausdorff topological vector space:
			\begin{enumerate}
				\item $\varphi$ is weakly continuous.
				\item For all polytopes $P_1,\dots,P_N\in\calP^n$, the map
				\begin{align*}
					(\lambda_1,\dots,\lambda_N)\mapsto\varphi\left(\sum_{j=1}^N\lambda_jP_j\right)
				\end{align*}
				is an $\R$-polynomial in $\lambda_1,\dots,\lambda_N\ge 0$.
			\end{enumerate}
		\end{theorem}
		\begin{remark}
			McMullen's proof uses that (2) implies that certain one-sided derivatives of the map in Eq.~\eqref{eq:mapWeakContinuity} exist. However, the existence of these derivatives only implies separate continuity of this map. Due to \autoref{theorem:equivJointSepCont}, this is sufficient to show that $\varphi$ is weakly continuous, however, this seems to be a nontrivial step required for the argument.
		\end{remark}
		
		The following was shown by Hadwiger in \cite{HadwigerVorlesungenuberInhalt1957} (see also  \cite{SchneiderConvexbodiesBrunn2014}*{Theorem~6.4.3}, or \cite{AleskerIntroductiontheoryvaluations2018}*{Theorem~3.1.1.} for a proof for continuous valuations on convex bodies that holds verbatim in the more general setting below).
		\begin{theorem}
			\label{theorem:HadwigerTopDegree}
			Let $E$ be a real Hausdorff topological vector space and $\varphi:\calP^n\rightarrow E$ a weakly continuous and translation invariant valuation of degree $n$. Then there exists $c\in E$ such that $\varphi=cV_n$.
		\end{theorem}
	
	\subsection{Measurable or locally bounded valuations}
	We combine the results of the previous two subsections to obtain the required properties of measurable valuations. Since the argument applies similarly to locally bounded valuations, we include this case in results. Throughout this section, $E$ denotes a finite dimensional real vector space.

	\begin{proposition}\label{proposition:homDecompMeasuBounded}
		Let $\varphi:\calP^n\rightarrow E$ be a translation invariant valuation and $\varphi=\sum_{j=0}^n\varphi_j$ be its rational homogeneous decomposition. Assume that $\varphi$ is either
		\begin{enumerate}
			\item measurable, or
			\item locally bounded.
		\end{enumerate}
		Then $\varphi_j$ is real homogeneous for every $0\le j\le n$.
	\end{proposition}
	\begin{proof}
		This is clear for $j=0$ since $\varphi_0$ is a multiple of the Euler characteristic, i.e. constant on $\calP^n$. Assume that $j=1$. For every $P\in\calP^n$, the map
		\begin{align*}
			f:[0,\infty)\rightarrow E,
			t\mapsto \varphi_1(tP)
		\end{align*}
		is additive, since for $s,t\ge0$,
		\begin{align*}
			\varphi_1((s+t)P)=\varphi_1(sP+tP)=\varphi_1(sP)+\varphi_1(tP)
		\end{align*}
		as $\varphi_1$ is Minkowski additive by Corollary \ref{corollary:MinkwskiAdditive1Hom}. Moreover, if $\varphi$ is measurable or locally bounded, then so is $\varphi_1$ by Corollary \ref{corollary:PropertiesHomComponents}. Thus, if $\varphi$ is locally bounded, then $f(t)=\varphi_1(tP)$ is a locally bounded solution of Cauchy's functional equation $f(s+t)=f(s)+f(t)$ for $s,t\ge 0$ and therefore a real linear function, i.e. there is $c\in E$ such that $\varphi(tP)=f(t)=ct$ for every $t\ge 0$.
		
		Next, note that the map $t\mapsto tP$ defines a continuous curve in $\calP^n$, so the map $t\mapsto f(t)=\varphi_1(tP)$ is measurable if $\varphi$ is measurable. In this case, $f$ is a measurable solution of Cauchy's functional equation and therefore a real linear function, i.e. $\varphi(tP)=f(t)=ct$ for every $t\ge 0$ for some $c\in\R$. This completes the case $k=1$.\\
		
		In the general case, note that
		\begin{align*}
			\varphi_j(tP)=\bar{\varphi_j}(tP,\dots,tP)
		\end{align*}
		for every $P\in\calP^n$, $t\ge 0$. Since the polarization $\bar{\varphi_j}$ is measurable or locally bounded if $\varphi_j$ has these properties, $\bar{\varphi_j}$ is a translation invariant and $1$-homogeneous measurable or locally bounded valuation in each argument. Thus the case $j=1$ implies that it is real $1$-homogeneous in each argument, and we obtain
		\begin{align*}
			\varphi_j(tP)=t^j\bar{\varphi_j}(P,\dots,P)=t^j\varphi_j(P)
		\end{align*}
		for every $P\in\calP^n$ and $t\ge 0$. Thus $\varphi_j$ is real $j$-homogeneous for every $2\le j\le n$.
	\end{proof}

	\begin{corollary}\label{corollary:measurableImpliesWeaklyCont}
		Let $\varphi:\calP^n\rightarrow E$ be a translation invariant valuation. If $\varphi$ is either measurable or locally bounded, then $\varphi$ is weakly continuous.
	\end{corollary}
	\begin{proof}
		In both cases, the map $(P_1,\dots,P_N)\mapsto \varphi(P_1+\dots+P_N)$ is either a measurable or locally bounded valuation in each argument. Thus $(\lambda_1,\dots,\lambda_N)\mapsto \varphi(\sum_{j=1}^N\lambda_jP_j)$ is an $\R$-polynomial in $\lambda_1,\dots,\lambda_N\ge 0$ by \autoref{proposition:homDecompMeasuBounded}. The claim follows from \autoref{theorem:RealHomogeneity}.
	\end{proof}

\section{Simple weakly continuous valuations of degree $1$}
	\label{section:CharacterizationSimpleWeak}
	The following two sections contain an improved version of a special case of Hadwiger's characterization of weakly continuous valuations from \cite{HadwigerTranslationsinvarianteadditiveund1952} for homogeneous valuations of degree $1$. In this case, Hadwiger's original proof by induction simplifies slightly, which we use to track how the desired representation formula is obtained in the induction step in order to show that it can be obtained by evaluating the valuation in a family of simplices.
	
	For $n\ge 2$ let $\calU^{n-1}_n$ denote the Stiefel manifold of all $(n-1)$-frames of ordered orthonormal tuples $u=(u_1,\dots,u_{n-1})$. We have a natural action of $\SO(n)$ on this space by left multiplication: $gu:=(gu_1,\dots,gu_{n-1})$ for $g\in\SO(n)$, $u\in\calU^{n-1}_n$. We will identify $\calU^{n-1}_n$ with $\SO(n)$ using the following map.
\begin{lemma}
	\label{lemma:isomorphismStiefelSOn}
	For $n\ge2$, the map
	\begin{align*}
		\SO(n)&\rightarrow \calU^{n-1}_n\\
		(u_1,\dots,u_n)&\mapsto (u_1,\dots,u_{n-1})
	\end{align*}
	is a well defined homeomorphism which commutes with the left action of $\SO(n)$ on both spaces. In particular, $\SO(n)$ operates transitively on $\calU^{n-1}_n$.
\end{lemma}
\begin{proof}
	The map is obviously well defined, continuous, and commutes with the action of $\SO(n)$ if we let $\SO(n)$ act on itself by left multiplication. The inverse map can be obtained by completing $u=(u_1,\dots,u_{n-1})$ to an oriented orthonormal basis.
\end{proof}

\subsection{A family of orthogonal simplices}

	For $u_0\in S^{n-1}$ consider the set $\calU^{n-1}_{n,u_0}:=\{u\in \calU^{n-1}_n: \langle u_0,u_j\rangle>0, 1\le j\le n-1\}$. For $u\in\calU^{n-1}_{n,u_0}$ and $1\le k\le n$, let 
	\begin{align*}
		P_{u,u_0}^k:&\R^n\rightarrow\left(\bigcap_{i=1}^{n-k} u_i^\perp\right)\cap u_0^\perp,\quad k\ne n,\\
		P_{u,u_0}^n:&\R^n\rightarrow u_0^\perp	
	\end{align*}
 denote the orthogonal projections. Let $u_n\in S^{n-1}$ be the unique vector such that $(u_1,\dots,u_{n})$ is a positively oriented basis of $\R^n$. We then inductively define 
	\begin{align*}
		&Q_{u_0}^1(u):= \sign(\langle u_{n},u_0\rangle)[0,u_n],\\
		&Q_{u_0}^k(u):=\conv(P_{u,u_0}^{k}(Q_{u_0}^{k-1}(u))\cup Q_{u_0}^{k-1}(u)), ~2\le k\le n,
	\end{align*}
	and set $Q_{u_0}(u):=Q_{u_0}^{n}(u)$. The following properties follow immediately from the definition.
	\begin{lemma}\label{lemma:PropertiesQ}
		Let $n\ge 2$ and $u\in \calU^{n-1}_{n,u_0}$. Then 
		\begin{align*}
			Q_{u_0}^k(u)=\sign(\langle u_n,u_0\rangle)[0, u_n,P_{u,u_0}^2(u_n),...,P_{u,u_0}^k(u_n)].
		\end{align*} In particular, the following holds:
		\begin{enumerate}
			\item $Q_{u_0}(u)$ is an orthogonal simplex of dimension $n$ for $\langle u_n,u_0\rangle\ne 0$.
			\item If $\langle u_n,u_0\rangle= 0$, then $Q_{u_0}(u)=0$.
			\item $Q_{u_0}(u)$ is contained in $\{x\in\R^n:\langle x,u_0\rangle\ge 0\}$.
		\end{enumerate}
	\end{lemma}
	
	We extend the map $Q_{u_0}:\calU^{n-1}_{n,u_0}\rightarrow\calP^n$ to $\calU^{n-1}_n$ by setting
	\begin{align*}
		Q_{u_0}(u):=Q_{u_0}(\sign(\langle u_0,u_1\rangle)u_1,\dots,\sign(\langle u_0,u_{n-1}\rangle)u_{n-1})
	\end{align*}
	if $\langle u_j,u_0\rangle\ne 0$ for all $1\le j\le n-1$, and $Q_{u_0}(u)=0$ otherwise.
	
	\begin{lemma}\label{lemma:restrQCont}
		The restriction of the map $Q_{u_0}:\calU^{n-1}_n\rightarrow\calP^n$ to 
		\begin{align*}
			\{u\in\SO(n)\cong\calU^{n-1}_n:\langle u_0,u_j\rangle \ne 0 ~\text{for}~1\le j\le n\}
		\end{align*}
		is continuous.
	\end{lemma}
	\begin{proof}
		This follows since $Q_{u_0}$ is given by
		\begin{align*}
			Q_{u_0}(u)=\pm \sign(\langle u_n,u_0\rangle)[0, u_n,P^2_{u,u_0}(u_n),\dots,P^{n}_{u,u_0}(u_n)]
		\end{align*}
		on the connected components of this set, where the projections depend continuously on $u\in \SO(n)$.
	\end{proof}

\subsection{Representation of simple weakly continuous valuations of degree $1$}
Throughout this section, we assume that $E$ is a Hausdorff real topological vector space. 
\begin{lemma}
	\label{lemma:vanishingCylinder}
	Let $n\ge 2$ and $\varphi:\calP^n\rightarrow E$ be a translation invariant valuation that is rational homogeneous of degree $1$. If $\varphi$ is simple, then $\varphi$ vanishes on cylinders.
\end{lemma}
\begin{proof}
	Every cylinder is of the form $P+I$, where $P$ is a polytope of dimension at most $n-1$ and $I$ is an interval. Since $\varphi$ is Minkowski additive by Corollary \ref{corollary:MinkwskiAdditive1Hom}, we obtain
	\begin{align*}
		\varphi(P+I)=\varphi(P)+\varphi(I).
	\end{align*}
	However, $P$ and $I$ are both lower dimensional, so the right hand side vanishes if $\varphi$ is simple. The claim follows.
\end{proof}

\begin{definition}
	For $u_0\in S^{n-1}$, we call $P\in\calP^n$ a $u_0$-skew cylinder if there exists $u\in S^{n-1}$ and a polytope $Q\in\calP(u^\perp)$ such that
	\begin{align*}
		P=\conv(\pi_{u_0^\perp}(Q)\cup Q)+x
	\end{align*}
	for some $x\in\R^n$.
\end{definition}

\begin{lemma}
	\label{lemma:DecompSkewCylinders}
	Let $u_0\in S^{n-1}$. For every polytope $P\in\calP^n$ there exist $u_0$-skew cylinders $C_1,\dots,C_N$ such that 
	\begin{align*}
		P\sqcup\bigsqcup_{j=1}^N C_j
	\end{align*}
	is a $u_0$-skew cylinder.
\end{lemma}
\begin{proof}
	Assume that $P\subset \{x\in\R^n:a\le \langle x,u_0\rangle\le b\}$. Let $F_1^+,\dots,F_N^+$ and $F_1^-,\dots,F_M^-$ be the faces of $P$ such that the scalar product of their outer unit normal with $u_0$ is either strictly positive or strictly negative. If $\pi_a:\R^n\rightarrow au_0+u_0^\perp$ and $\pi_b:\R^n\rightarrow bu_0+u_0^\perp$ denote the orthogonal projections, then
	\begin{align*}
		[0,b-a]u_0+\pi_a(P)=P\sqcup\bigsqcup_{j=1}^N C_j^+\sqcup\bigsqcup_{j=1}^M C^-_j
	\end{align*}
	for the $u_0$-skew cylinders
	\begin{align*}
		&C_j^+=\conv(\pi_b(F_j^+)\cup F_j^+), &C_j^-=\conv(\pi_a(F_j^-)\cup F_j^-).
	\end{align*}
\end{proof}

\begin{definition}\label{def:Phi_u}
	Let $\varphi:\calP^n\rightarrow E$ be a function and fix $u_0\in S^{n-1}$. For every $u\in S^{n-1}$ we define  a functional $\varphi^{u_0}_u:\calP(u^\perp)\rightarrow E$ as follows: 
	\begin{enumerate}
		\item If $\langle u,u_0\rangle=0$, we set $\varphi^{u_0}_u=0$.
		\item If $\langle u,u_0\rangle>0$, then we choose for every $Q\in \calP(u^\perp)$ a vector $x_Q\in u^\perp$ such that $Q+x_Q\subset \{x\in \R^n: \langle x,u_0\rangle\ge 0\}$ and set
		\begin{align*}
			\varphi^{u_0}_u(Q):=\varphi\left(\conv\left(\pi_{u_0^\perp}(Q+x_Q)\cup (Q+x_Q)\right)\right).
		\end{align*}
		\item If $\langle u,u_0\rangle<0$, we define
		\begin{align*}
			\varphi^{u_0}_u(Q):=-\varphi^{u_0}_{-u}(Q)\quad\text{for}~Q\in\calP(u^\perp).
		\end{align*}
	\end{enumerate}
\end{definition}

\begin{lemma}
	Let $n\ge 2$ and $\varphi:\calP^n\rightarrow E$ be a weakly continuous and translation invariant valuation homogeneous of degree $1$, and assume that $\varphi$ is simple. For every $u\in S^{n-1}$, $\varphi^{u_0}_u$ is a weakly continuous, $1$-homogeneous, translation invariant, simple valuation. Moreover, $\varphi^{u_0}_u=0$ if $\langle u_0,u\rangle=\pm 1$.
\end{lemma}
\begin{proof}
	Note that it is enough to show the result for all $u\in S^{n-1}$ with $\langle u,u_0\rangle>0$. Let us first show that for $Q\in \calP(u^\perp)$, the value of 
	\begin{align}
		\label{eq:constructionPhi_u}
		\varphi^{u_0}_u(Q):=\varphi\left(\conv\left(\pi_{u_0^\perp}(Q+x_Q)\cup (Q+x_Q)\right)\right).
	\end{align}
	does not depend on the choice of $x_Q\in u^\perp$ with $Q+x_Q\subset \{x\in \R^n: \langle x,u_0\rangle\ge 0\}$. If $x_Q'\in u^\perp$ is another such vector, then we may without loss of generality assume that $x'_Q=x_Q+x$ for some $x\in u^\perp$ with $\langle x,u_0\rangle\ge0$. We write $x=\tilde{x}+\langle x,u_0\rangle u_0$ for $\tilde{x}\in u_0^\perp$. Then
	\begin{align*}
		&\conv\left(\pi_{u_0^\perp}(Q+x'_Q)\cup (Q+x'_Q)\right)=\conv\left([\pi_{u_0^\perp}(Q+x_Q)+\tilde{x}]\cup [(Q+x_Q)+x]\right)\\
		=&\conv\left([\pi_{u_0^\perp}(Q+x_Q)+x]\cup [(Q+x_Q)+x]\right)\sqcup\left(\pi_{u_0^\perp}(Q+x_Q)\times[0,\langle x,u_0\rangle u_0]\right),
	\end{align*}
	where the second set is a cylinder. Since $\varphi$ is a simple and $1$-homogeneous valuation, it vanishes on cylinders by \autoref{lemma:vanishingCylinder}. Thus the valuation property implies
	\begin{align*}
		&\varphi\left(\conv\left(\pi_{u_0^\perp}(Q+x'_Q)\cup (Q+x'_Q)\right)\right)\\
		&\quad=\varphi\left(\conv\left(\pi_{u_0^\perp}(Q+x_Q)\cup (Q+x_Q)\right)+x\right)\\
		&\quad=\varphi\left(\conv\left(\pi_{u_0^\perp}(Q+x_Q)\cup (Q+x_Q)\right)\right),
	\end{align*}
	where we used that $\varphi$ is translation invariant in the last step. Thus Eq.~\eqref{eq:constructionPhi_u} does not depend on the choice of $x_Q\in u^\perp$ with $Q+x_Q\subset \{x\in \R^n: \langle x,u_0\rangle\ge 0\}$. From this one immediately deduces that $\varphi^{u_0}_u$ is a translation invariant and weakly continuous valuation on $\calP(u^\perp)$.
	
	In order to see that $\varphi^{u_0}_u$ is simple, note that for any $Q\in\calP(u^\perp)$ contained in an $(n-2)$-dimensional subspace $H\subset u^\perp$, we may choose $x_Q\in H$ such that $Q+x_Q\subset \{x\in \R^n: \langle x,u_0\rangle\ge 0\}$. In particular the set $ \conv(\pi_{u_0^\perp}(Q+x_Q)\cup (Q+x_Q))$ is contained in $H+\R u_0$, which is a subspace of dimension at most $n-1$, so $\varphi$ vanishes on this set. Thus $\varphi^{u_0}_u(Q)=0$, and we see that $\varphi^{u_0}_u$ is simple. Finally, let us show that $\varphi^{u_0}_u$ is $ 1$-homogeneous. For $t\ge 0$, the set $tQ+tx_Q$ is contained in $\{x\in \R^n: \langle x,u_0\rangle\ge 0\}$, so by definition
	\begin{align*}
		\varphi^{u_0}_u(tQ)=&\varphi\left(\conv\left(\pi_{u_0^\perp}(tQ+tx_Q)\cup (tQ+tx_Q)\right)\right)\\
		=&t\varphi\left(\conv\left(\pi_{u_0^\perp}(Q+x_Q)\cup (Q+x_Q)\right)\right)
	\end{align*}
	since $\varphi$ is $1$-homogeneous. Thus $\varphi^{u_0}_u(tQ)=t\varphi^{u_0}_u(Q)$, and we see that $\varphi^{u_0}_u$ is $1$-homogeneous.
\end{proof}
\begin{proposition}
	\label{proposition:HadwigerWeakSimple1HomCase}
	Let $n\ge 2$ and $\varphi:\calP^n\rightarrow E$ be a weakly continuous, translation invariant, $1$-homogeneous, and simple valuation. Fix $u_0\in S^{n-1}$. Then
	\begin{align*}
		\varphi(P)=\sum_{u\in N(P)}\varphi^{u_0}_{u}(F(P,u))
 	\end{align*}
 	for all $P\in\calP^n$, where $N(P)$ denotes the set of outer unit normals to the faces of $P$ and $F(P,u)=\{x\in P: \langle x,u\rangle=h_P(u)\}$ is the face in direction $u$.
\end{proposition}
\begin{proof}
	Set 
	\begin{align*}
		\tilde{\varphi}(P):=\sum_{u\in N(P)}\varphi^{u_0}_{u}(F(P,u))=\sum_{u\in S^{n-1}}\varphi^{u_0}_{u}(F(P,u))
	\end{align*} for $P\in\calP^n$, where the second equation follows since $\varphi^{u_0}_u$ is simple, so only a finite number of terms appear in the sum. Then $\tilde{\varphi}$ is a weakly continuous and simple valuation and by construction, $\varphi$ coincides with $\tilde{\varphi}$ on all $u_0$-skew cylinders. By \autoref{lemma:DecompSkewCylinders}, there exist $u_0$-skew cylinders $C_1,\dots,C_N$ such that 
	\begin{align*}
		P\sqcup\bigsqcup_{j=1}^N C_j
	\end{align*}
	is a $u_0$-skew cylinder. Since $\varphi$ and $\tilde{\varphi}$ are both simple valuations, we therefore obtain
	\begin{align*}
		\varphi(P)=\varphi\left(P\sqcup\bigsqcup_{j=1}^N C_j\right)-\sum_{j=1}^N\varphi(C_j)=\tilde{\varphi}
		\left(P\sqcup\bigsqcup_{j=1}^N C_j\right)-\sum_{j=1}^N\tilde{\varphi}(C_j)=\tilde{\varphi}(P),
	\end{align*}
	which shows the desired formula. Note that we used the Inclusion-Exclusion Principle from \autoref{theorem:InclusionExclusionPrinciple} in this step.
\end{proof}

For a polytope $P\in\calP^n$ and $u\in S^{n-1}$, we define \begin{align*}
	F(P,u)=\{x\in P: \langle x,u\rangle=h_P(u)\}
\end{align*} to be the face in direction $u$. If $u=(u_1,\dots,u_{n-1})\in\calU^{n-1}_n$, we inductively set
\begin{align*}
	F(P,(u_1,\dots,u_k)):=F(F(P,(u_1,\dots,u_{k-1})),u_{k})\quad \text{for}~2\le k\le n-1.
\end{align*}
Then $u\in\calU^{n-1}_n$ is called \emph{$P$-tight} if $F(P,(u_1,\dots,u_k))$ is an $(n-k)$-dimensional polytope for all $1\le k\le n-1$. We denote by $\calU^{n-1}_n(P)$ the finite set of $P$-tight elements in $\calU^{n-1}_n$.

The following result provides the explicit form of Hadwiger's characterization from \cite{HadwigerTranslationsinvarianteadditiveund1952} for $1$-homogeneous valuations mentioned in the introduction.
\begin{theorem}
	\label{theorem:1HomSimple}
	Let $n\ge 2$ and $\varphi:\mathcal{P}^n\rightarrow E$ be a translation invariant, simple, and weakly continuous valuation homogeneous of degree $1$. For $u_0\in S^{n-1}$ define $f:\calU^{n-1}_n\rightarrow E$ by
		\begin{align*}
		f(u)=\left(\prod_{j=1}^{n-1}\sign(\langle u_0,u_j\rangle)\right)\varphi(Q_{u_0}(u)).
	\end{align*}
	Then for every $P\in\calP^n$, 
	\begin{align}
		\label{eq:formulaWeakSimpleRepresentation}
		\varphi(P)=&\sum_{u\in \mathcal{U}^{n-1}_n(P)}f(u)\vol_{1}(F(P,u)).
	\end{align}
\end{theorem}
\begin{proof}
	We establish \eqref{eq:formulaWeakSimpleRepresentation} by induction on the dimension. Thus let $n=2$ and $\varphi:\calP^2\rightarrow E$ a weakly continuous, simple and $1$-homogeneous valuation. By \autoref{proposition:HadwigerWeakSimple1HomCase}, we have
	\begin{align*}
		\varphi(P)=\sum_{u\in N(P)}\varphi^{u_0}_{u}(F(P,u)),
	\end{align*}
	where $\varphi^{u_0}_u$ is a $1$-homogeneous and weakly continuous valuation on $u^\perp$. By \autoref{theorem:HadwigerTopDegree}, $\varphi^{u_0}_u=c(u)\vol_1$ for a constant $c(u)\in E$. But then
	\begin{align*}
		c(u)=\varphi^{u_0}_u([0,u^\perp])
	\end{align*}
	for the unique $u^\perp\in S^1$ such that $(u,u^\perp)\in \SO(2)$. From the definition of $Q_{u_0}$ and $\varphi^{u_0}_u$, we thus obtain for $\langle u_0,u\rangle \ne \pm 1$
	\begin{align*}
		&c(u)=\varphi^{u_0}_u([0,u^\perp])\\
		=&\sign(\langle u_0,u\rangle)\varphi\left(\conv\left(\pi_{u_0^\perp}(\sign(\langle u^\perp,u_0\rangle)[0,u^\perp])\cup \sign(\langle u^\perp,u_0\rangle)[0,u^\perp]\right)\right)\\
		=&\sign(\langle u_0,u\rangle)\varphi(Q_{u_0}(u)).
	\end{align*}
	For $\langle u_0,u\rangle= \pm 1$, $\varphi^{u_0}_u=0$ and $Q_{u_0}=0$ since $\langle u_0,u^\perp\rangle=0$ in this case. Thus for every $u\in S^1$,
	\begin{align*}
		c(u)=\sign(\langle u_0,u\rangle)\varphi(Q_{u_0}(u)),
	\end{align*}
	and we obtain
	\begin{align*}
		\varphi(P)=\sum_{u\in \calU^1_2(P)}\sign(\langle u_0,u\rangle)\varphi(Q_{u_0}(u))\vol_1(F(P,u)),
	\end{align*}
	which shows the claim for $n=2$.
	
	Now assume that the claim holds in all dimensions up to $n-1$, $n\ge 3$. If $\varphi:\calP^n\rightarrow E$ is a weakly continuous, translation invariant, simple and $1$-homogeneous valuation, we may again write
	\begin{align*}
		\varphi(P)=\sum_{u_1\in N(P)}\varphi^{u_0}_{u_1}(F(P,u_1)),
	\end{align*}
	where $\varphi^{u_0}_{u_1}$ is defined as in \autoref{def:Phi_u}. Then $\varphi^{u_0}_{u_1}$ is a weakly continuous, translation invariant, simple and $1$-homogeneous valuation on $\calP(u_1^{\perp})$. Moreover, $\varphi^{u_0}_{u_1}=0$ if $\langle u_0,u_1\rangle\in\{-1,0,1\}$. Thus assume that $\langle u_0,u_1\rangle\notin\{-1,0,1\}$. Then the projection of $u_0$ onto $u_1^{\perp}$ is nontrivial and we let $\tilde{u}_0\in u_1^{\perp}$, $|\tilde{u}_0|=1$, denote the normalized projection of $u_0$ onto $u_1^\perp$. The induction assumption applied to $\varphi^{u_0}_{u_1}$ shows that for $\tilde{P}\in\calP(u_1^{\perp})$,
	\begin{align}\label{eq:inductionRep1Hom}
		\varphi^{u_0}_{u_1}(\tilde{P})=\sum_{\substack{\hat{u}\in \calU^{n-2}_{u_1^{\perp}}(\tilde{P}),\\ \hat{u}=(u_2,\dots,u_{n-1})}}\left(\prod_{j=2}^{n-1}\sign(\langle \tilde{u}_0,u_j\rangle)\right)\varphi^{u_0}_{u_1}(\tilde{Q}_{\tilde{u}_0}(\hat{u}))\vol_1(F(\tilde{P},\hat{u})),
	\end{align}
	where $\calU^{n-2}_{u_1^{\perp}}$ denotes the Stiefel manifold of $(n-2)$-frames in $u_1^{\perp}$, $\calU^{n-2}_{u_1^{\perp}}(\tilde{P})$ the set of $\tilde{P}$-tight elements, and $\tilde{Q}_{\tilde{u}_0}(\hat{u})\subset u_1^{\perp}$ is the simplex defined as in \autoref{lemma:PropertiesQ}. 
	
	Note that for $2\le j\le n-1$, we have
	\begin{align}
		\label{eq:sameSigns}
		\sign(\langle u_j,\tilde{u}_0\rangle)=\sign(\langle u_j,u_0\rangle)
	\end{align}
	as $\tilde{u}_0$ is just the normalized projection of $u_0$ onto $u_1^\perp$ and $u_j$ is orthogonal to $u_1$ for $2\le j\le n-1$.
	
	Let $u_n\in u_1^{\perp}$ be the unique vector such that $(u_1,u_2,\dots,u_n)$ is an orthonormal basis of $\R^n$ and set $u=(u_1,\dots,u_{n-1})\in \calU^{n-1}_n$. \autoref{lemma:PropertiesQ} shows that
	\begin{align*}
		\tilde{Q}_{\tilde{u}_0}(\hat{u})=\sign(\langle u_0,u_n\rangle)[0,u_n,\tilde{P}^2_{\hat{u},\tilde{u}_0}(u_n),\dots,\tilde{P}^{n-1}_{\hat{u},\tilde{u}_0}(u_n)]
	\end{align*}
	where $\tilde{P}^k_{\hat{u},\tilde{u}_0}:u_1^{\perp}\rightarrow \left(\bigcap_{i=2}^{n-k} u_i^\perp\right)\cap \tilde{u}_0^\perp\subset u_1^{\perp}$ is the orthogonal projection. In particular, for $k=2,\dots, n-1$,
	\begin{align*}
		P^k_{u,u_0}:\R^n\rightarrow \left(\bigcap_{i=1}^{n-k} u_i^\perp\right)\cap u_0^\perp=u_1^{\perp}\cap\left(\bigcap_{i=2}^{n-k} u_i^\perp\right)\cap \tilde{u}_0^\perp
	\end{align*}
	coincides with $\tilde{P}^k_{\hat{u},\tilde{u}_0}$ on $u_1^{\perp}$ since both are orthogonal projections onto the same subspace.
	If $\pi_{u_0}$ denotes the orthogonal projection onto $u_0^\perp$, this implies for $2\le k\le n-1$,  
	\begin{align*}
		\pi_{u_0}(\tilde{P}^{k}_{\hat{u},\tilde{u}_0}(u_n))=\tilde{P}^{k}_{\hat{u},\tilde{u}_0}(u_n)=P^{k}_{u,u_0}(u_n),
	\end{align*}
	since this point is already contained in $u_0^\perp$, while 
	\begin{align*}
		\pi_{u_0}(u_n)=P^n_{u,u_0}(u_n)
	\end{align*}
	by the definition of $P^n_{u,u_0}$. Thus
	\begin{align*}
		&\mathrm{conv}\left(\pi_{u_0}(\tilde{Q}_{\tilde{u}_0}(\hat{u}))\cup \tilde{Q}_{\tilde{u}_0}(\hat{u})\right)\\
		=&\sign(\langle u_0,u_n\rangle)[0,u_n,P^2_{u,u_0}(u_n),\dots,P^n_{u,u_0}(u_n)]=Q_{u_0}(u),
	\end{align*}
	and from the definition of $\varphi^{u_0}_{u_1}$, we thus obtain
	\begin{align*}
		\varphi^{u_0}_{u_1}(\tilde{Q}_{\tilde{u}_0}(\hat{u}))=\sign(\langle u_0,u_1\rangle)\varphi(Q_{u_0}(u))
	\end{align*}
	Combining this equation with Eq.~\eqref{eq:inductionRep1Hom} and Eq.~\ref{eq:sameSigns}, we obtain the desired representation.
\end{proof}

The benefit of this representation is that we can use it to transfer properties of $\varphi$ to the function $f$.
\begin{corollary}
	\label{corollary:propertiesF}
	Let $E$ be a finite dimensional real vector space and $\varphi:\calP^n\rightarrow E$ a weakly continuous, translation invariant, $1$-homogeneous, simple valuation, and let $f:\calU^{n-1}_n\rightarrow E$ be defined as in \autoref{theorem:1HomSimple}. Then the following holds:
	\begin{enumerate}
		\item If $\varphi$ is measurable, then $f$ is measurable.
		\item If $\varphi$ is bounded on bounded families of orthogonal simplices, then $f$ is bounded.
	\end{enumerate}
\end{corollary}
\begin{proof}
	The second property follows directly from the definition of $f$, since $Q_{u_0}(u)$ is an orthogonal simplex with diameter bounded independent of $u\in \calU^{n-1}_n$. In order to see that $f$ is measurable if $\varphi$ is measurable, observe that the restriction of $f$ to the open set 
	\begin{align*}
		\{u\in\SO(n):\langle u_0,u_j\rangle \ne 0 ~\text{for}~1\le j\le n\}\subset \SO(n)\cong \calU^{n-1}_n
	\end{align*} 
	is the composition of $\varphi$ with the continuous function $Q_{u_0}$, compare \autoref{lemma:restrQCont}, and so in particular measurable, and that $f$ is zero on the complement.
\end{proof}

\subsection{Equivariant simple and measurable valuations of degree $1$}
For this section, assume that $E$ is a finite dimensional real vector space. We start with two simple observations.
\begin{lemma}
	\label{lemma:Equivariance1HomSimple}
	Let $f:\calU^{n-1}_n\rightarrow E$ be a function and define $\varphi:\calP^n\rightarrow E$ by 
	\begin{align*}
		\varphi(P):=\sum_{u\in \calU^{n-1}_n(P)}f(u)\vol_{1}(F(P,u)).
	\end{align*}
	Then for any $g\in\SO(n)$,
	\begin{align*}
		\varphi(gP)=\sum_{u\in \calU^{n-1}_n(P)}f(gu)\vol_{1}(F(P,u)).
	\end{align*}
\end{lemma}
\begin{proof}
	If $u\in \calU^{n-1}_n$ is $gP$-tight, then it is easy to see that $g^Tu$ is $P$-tight and $F(gP,u)=gF(P,g^Tu)$. Thus	
	\begin{align*}
		\varphi(gP)=\sum_{u\in \calU^{n-1}_n(gP)}f(u)\vol_{1}(F(gP,u))
		=\sum_{g^Tu\in \calU^{n-1}_n(P)}f(u)\vol_{1}(F(P,g^Tu)).
	\end{align*}
	Replacing $u$ by $gu$ in the sum, we obtain the desired result.
\end{proof}

The next result is essentially trivial, however, since it is central to the following regularity result, we include the argument. 
\begin{lemma}\label{lemma:averageContinuous}
	Let $(E,\pi_E)$ be a finite dimensional representation of $\SO(n)$ and $f\in L^1(\SO(n),E)$. Then the pointwise defined function
	\begin{align*}
		\tilde{f}(h):=\int_{\SO(n)} \pi_E(g)f(g^{-1}h)d\mu_{\SO(n)}(g)
	\end{align*}
	is continuous.
\end{lemma}
\begin{proof}
	Since the Haar measure is left invariant, we have
	\begin{align*}
		\tilde{f}(h)=&\int_{\SO(n)} \pi_E(hg)f(g^{-1})d\mu_{\SO(n)}(g)\\
		=&\pi_E(h)\left(\int_{\SO(n)} \pi_E(g)f(g^{-1})d\mu_{\SO(n)}(g)\right)=\pi_E(h)\tilde{f}(e).
	\end{align*}
	Since $\SO(n)$ operates continuously on $E$, the result follows.
\end{proof}

Given $u\in \calU^{n-1}_n$ and $e=(\epsilon_1,\dots,\epsilon_{n-1})\in \{\pm 1\}^{n-1}$, we set \begin{align*}
	eu:=(\epsilon_1 u_1,\dots,\epsilon_{n-1}u_{n-1})
\end{align*} and $\pi(e)=\epsilon_1\dots\epsilon_{n-1}$. We then call a function $f:\calU^{n-1}_n\rightarrow E$ into a real vector space odd if 
\begin{align*}
	f(eu)=\pi(e)f(u)
\end{align*}
for all $e\in \{\pm1\}^{n-1}$ and $u\in\calU^{n-1}_n$. Note that the action of $\{\pm 1\}^{n-1}$ on $\calU^{n-1}_n$ commutes with the natural operation of $\SO(n)$ on this space, i.e. for $g\in \SO(n)$, we have 
\begin{align*}
	g(eu)=e(gu)\quad \text{for}~u\in \calU^{n-1}_n, g\in\SO(n), e\in \{\pm1\}^{n-1}.
\end{align*}

Recall that the homeomorphism $\SO(n)\cong \calU^{n-1}_n$ from \autoref{lemma:isomorphismStiefelSOn} commutes with the action of $\SO(n)$, so this map preserves the corresponding Haar measures on $\SO(n)$ and $\calU^{n-1}_n$. For clarity, we will distinguish the measures and denote them by $\mu_{\SO(n)}$ and $\mu_{\calU^{n-1}_n}$, respectively.

\begin{theorem}\label{theorem:Equiv1homValuation}
	Let $(E,\pi_E)$ be a finite dimensional representation of $\SO(n)$, $\varphi:\calP^n\rightarrow E$ an $\SO(n)$-equivariant, translation invariant, and measurable valuation that is simple and homogeneous of degree $1$. Then there exists an odd function $f\in C(\SO(n),E)^{\SO(n)}$ such that
	\begin{align*}
		\varphi(P)=\sum_{u\in \calU^{n-1}_n(P)}f(u)\vol_1(F(P,u)).
	\end{align*}
\end{theorem}
\begin{proof}
	We may equip $E$ with an $\SO(n)$-invariant inner product, so $\SO(n)$ acts by isometries on $E$. Let $f:\calU^{n-1}_n\rightarrow E$ be the function from \autoref{theorem:1HomSimple} such that
	\begin{align*}
		\varphi(P)=\sum_{u\in \calU^{n-1}_n(P)}f(u)\vol_1(F(P,u)).
	\end{align*}
	Since $Q_{u_0}(eu)=Q_{u_0}(u)$ by construction, $f$ is an odd function. Moreover, $f$ is measurable by \autoref{corollary:propertiesF}, so we may for every $\epsilon>0$ choose a measurable set $K_\epsilon\subset \calU^{n-1}_{n}$ with $\mu_{\mathcal{U}^{n-1}_n}(\calU^{n-1}_{n}\setminus K_\epsilon)\le \epsilon$ and such that $f|_{K_\epsilon}$ is bounded by a constant $M_\epsilon>0$.
	
	For $P\in\calP^n$ consider the set
	\begin{align*}
		G_\epsilon(P):=\{g\in\SO(n):g^{-1}u\in K_\epsilon~\text{for every}~u\in \calU^{n-1}_n(P)\}.
	\end{align*}
	Then $G_\epsilon(P)$ is measurable and
	\begin{align*}
		\mu_{\SO(n)}(G_\epsilon(P))=& 1- \mu_{\SO(n)}(\{g\in\SO(n):g^{-1}u\notin K_\epsilon~\text{for some}~u\in \calU^{n-1}_n(P)\})\\
		\ge& 1-\sum_{u\in\mathcal{U}^{n-1}_n(P)}\mu_{\SO(n)}(\left\{g\in\SO(n): g^{-1} u \in\calU^{n-1}_n\setminus K_\epsilon \right\})\\
		=&1-|\mathcal{U}^{n-1}_n(P)|\mu_{\mathcal{U}^{n-1}_n}(\calU^{n-1}_n\setminus K_\epsilon)\ge 1-\epsilon |\mathcal{U}^{n-1}_n(P)|.
	\end{align*}
	In particular, the Haar measure of this set is positive for $\epsilon<\frac{1}{|\mathcal{U}^{n-1}_n(P)|}$, so $G_\epsilon(P)$ is nonempty. In this case, choose $g_{\epsilon,P}\in G_\epsilon(P)$. We may then apply \autoref{lemma:Equivariance1HomSimple} to obtain
	\begin{align*}
		\varphi(P)=\pi_E(g_{\epsilon,P})\varphi(g_{\epsilon,P}^{-1}P)=\sum_{u\in\calU^{n-1}_n(P)}\pi_E(g_{\epsilon,P})f(g_{\epsilon,P}^{-1}u)\vol_1(F(P,u)),
	\end{align*}
	and since $g^{-1}_{\epsilon,P}u\in K_\epsilon$ for every $u\in\calU^{n-1}_n(P)$ by assumption, we obtain the inequality
	\begin{align*}
		|\varphi(P)|\le  M_\epsilon\sum_{u\in\calU^{n-1}_n(P)}\vol_1(F(P,u))\le M_\epsilon |\calU^{n-1}_n(P)| \diam(P)
	\end{align*}
	for $\epsilon<\frac{1}{|\mathcal{U}^{n-1}_n(P)|}$. Now let $P=\Delta$ be a simplex, so that $|\calU^{n-1}_n(\Delta)|=\frac{(n+1)!}{2}$. If we fix $\epsilon_0<\frac{2}{(n+1)!}$, this implies
	\begin{align*}
		|\varphi(\Delta)|\le M_{\epsilon_0}\frac{(n+1)!}{2} \diam(\Delta)
	\end{align*}
	for \emph{every} simplex $\Delta$. In particular, Corollary \ref{corollary:propertiesF} shows that the function $f$ is bounded. Since $f$ is measurable by \autoref{corollary:propertiesF}, $f$ is in particular integrable, i.e. $f\in L^1(\SO(n),E)$, and we may average the function $g\mapsto \varphi(P)=\pi_E(g)\varphi(g^{-1}P)$ over $g\in \SO(n)$ to obtain (using \autoref{lemma:Equivariance1HomSimple} again)
	\begin{align*}
		\varphi(P)=\sum_{u\in\calU^{n-1}_n(P)}\int_{\SO(n)}\pi_E(g)f(g^{-1}u)d\mu_{\SO(n)}(g)\cdot\vol_1(F(P,u)).
	\end{align*}
	The function $u\mapsto \int_{\SO(n)}\pi_E(g)f(g^{-1}u)d\mu_{\SO(n)}(g)$ on $\calU^{n}_{n-1}\cong \SO(n)$ is then $\SO(n)$-invariant and odd. Moreover, it is continuous due to \autoref{lemma:averageContinuous}. Thus it belongs to $C(\SO(n),E)^{\SO(n)}$, which completes the proof.
\end{proof}

The Canonical Decomposition for $C(\SO(n),E)$ can be used to obtain some more precise information on the space of these valuations. More precisely, \autoref{corollary:canonicalDecompTensor} shows that
\begin{align*}
C(\SO(n),E)^{\SO(n)}\cong \bigoplus_{[\pi]\in\widehat{\SO(n)}}E_\pi^*\otimes (E_\pi\otimes E)^{\SO(n)},
\end{align*}
where only a finite number of summands are nontrivial. In particular, Schur's Lemma implies the following estimate of the dimension of this space.
	\begin{corollary}\label{lemma:TrivialitySimpleValuations}
	Let $E$ be an irreducible representation of $\SO(n)$. The space of translation invariant, measurable, simple, $1$-homogeneous valuations $\varphi:\calP^n\rightarrow E$ that are $\SO(n)$-equivariant is at most $\dim E$-dimensional.
\end{corollary}
We will use a slightly more refined implication of the Canonical Decomposition in \autoref{section:SteinerPoint}.

\section{Rigid motion invariant measurable valuations on polytopes}
	\label{section:HadwigerPolytope}

	The key consequence of \autoref{theorem:Equiv1homValuation} is the following vanishing result for $1$-homogeneous $\SO(n)$-invariant valuations.
	\begin{proposition}
		\label{prop:simple1homHadwiger}
		Let $n\ge 2$. If $\varphi:\calP^n\rightarrow\R$ is a translation invariant, $\SO(n)$-invariant, measurable, and simple valuation that is homogeneous of degree $1$, then $\varphi=0$.
	\end{proposition}
	\begin{proof}
		By \autoref{theorem:Equiv1homValuation}, there exists an $\SO(n)$-invariant and odd function $f\in C(\SO(n))$ such that
		\begin{align*}
			\varphi(P)=\sum_{u\in \calU^{n-1}_n(P)}f(u)\vol_1(F(P,u)).
		\end{align*}
		for all $P\in\calP^n$. Since $f$ is $\SO(n)$-invariant, $f$ is constant. But for $n\ge 2$, every constant odd function on $\SO(n)$ vanishes identically. Thus $\varphi=0$.
	\end{proof}
	\autoref{prop:simple1homHadwiger} is the required replacement for the characterization of the mean width in Hadwiger's original proof of \autoref{theorem:HadwigerConvexBodies}. For completeness, we present a variant of the remaining arguments that explicitly exploits the polynomial behavior of translation invariant valuations in order to streamline the proof. We start by extending \autoref{prop:simple1homHadwiger} to arbitrary degrees of homogeneity $0\le k\le n-1$. This relies on the following result.
	
	\begin{lemma}[\cite{Chensimplifiedelementaryproof2004}*{Lemma 3.3}]\label{lemma:GeneralizedDecompOrthSimplices}
		Let $P\in\calP^n$ be given. There exist orthogonal simplices $S_1,\dots, S_m$ and $\epsilon_j\in\{\pm 1\}$ such that 
		\begin{align*}
			1_P-\sum_{j=1}^m\epsilon_j1_{S_j}
		\end{align*}
		is an integer combination of indicators of polytopes of dimension at most $n-1$. 
	\end{lemma}

	\begin{corollary}
		\label{corollary:SimpleVanishOrthogonalSimplices}
		If $\varphi:\calP^n\rightarrow\R$  is a simple valuation that vanishes on orthogonal simplices, then $\varphi=0$.
	\end{corollary}
	\begin{proof}
		If $P\in\calP^n$ and $S_1,\dots,S_m$ are the orthogonal simplices from \autoref{lemma:GeneralizedDecompOrthSimplices} with signs $\epsilon_j\in\{\pm1\}$, $1\le j\le m$, then the fact that $\varphi$ is simple implies that
		\begin{align*}
			\varphi(P)=\sum_{j=1}^m\epsilon_j \varphi(S_j)=0.
		\end{align*}
		Note that this step uses \autoref{corollary:Groemer}, i.e. it relies on the fact that $\varphi$ satisfies the full Inclusion-Exclusion Principle from \autoref{theorem:InclusionExclusionPrinciple}.
		
	\end{proof}

	\begin{proposition}
		\label{proposition:simpleValuationsVanish}
		Let $0\le k\le n-1$ and $\varphi:\calP^n\rightarrow\R$ be a measurable, translation and $\SO(n)$-invariant, as well as simple valuation that is homogeneous of degree $k$. Then $\varphi=0$.
	\end{proposition}
	\begin{proof}
		The case $k=0$ follows from the fact that a $0$-homogeneous valuation is a multiple of the Euler characteristic, i.e. constant on $\calP^n$, while the case $k=1$ is covered by \autoref{prop:simple1homHadwiger}. 
		
		In the remaining cases, we argue by induction on $n\ge2$, the case $n=2$ being trivial. Assume that the claim holds for all dimensions at most $n-1$. Let $\varphi:\calP^n\rightarrow\R$ be a $k$-homogeneous, measurable, translation and $\SO(n)$-invariant, as well as simple valuation with $2\le k\le n-1$.
		We claim that $\varphi$ vanishes on all direct orthogonal Minkowski sums. Let $\R^n=E\oplus F$ be an orthogonal direct sum decomposition with $1 \le \dim E\le n-1$. From the definition of the polarization, we have for $P_E\in\calP(E)$ and $P_F\in\calP(F)$,
		\begin{align*}\varphi(P_E+P_F)=\sum_{j=0}^k \binom{k}{j}\bar{\varphi}(P_E[j],P_F[k-j]).
		\end{align*}
		The maps $(P_E,P_F)\mapsto \bar{\varphi}(P_E[j],P_F[k-j])$ are measurable and rigid motion invariant valuations in each argument, compare Remark \ref{remark:propertiesPolarization}. Moreover, since $\varphi$ is simple, so is $(P_E,P_F)\mapsto \bar{\varphi}(P_E[j],P_F[k-j])$ in each argument. Note that these maps are of degree $j$ in the first argument and of degree $k-j$ in the second. The induction assumption thus implies that these maps vanish identically unless $j=\dim E$, $k-j=\dim F$, i.e. unless $k=\dim E+\dim F=n$. Since $k\le n-1$ by assumption, this case does not occur, so all of these maps vanish identically. We thus obtain $\varphi(P_E+P_F)=0$ for all $P_E\in\calP(E)$, $P_F\in\calP(F)$. Since this holds for all orthogonal decompositions $\R^n=E\oplus F$, we see that $\varphi$ vanishes on proper orthogonal Minkowski sums.
		
		Now let $\Delta$ be an orthogonal simplex. We use the Canonical Simplex Decomposition in \autoref{theorem:CanonicalSimplexDecomp} to write for $t\in(0,1)$
		\begin{align*}
			\Delta=\bigsqcup_{j=0}^n\left((1-t)\underline{S}_j+t\overline{S}_{n-j}\right).
		\end{align*}
		Since $\varphi$ is simple and homogeneous of degree $2\le k\le n-1$, this implies
		\begin{align*}
			\varphi(\Delta)=\left[(1-t)^k+t^k\right]\varphi(\Delta)+\sum_{j=1}^{n-1}\varphi((1-t)\underline{S}_j+t\overline{S}_{n-j}).
		\end{align*}
		However, the second term vanishes since, up to a translation, $(1-t)\underline{S}_j+t\overline{S}_{n-j}$ is a proper orthogonal Minkowski sum for $1\le j\le n-1$, compare \autoref{remark:simplexDecompOrthogonal}. We thus obtain
		\begin{align*}
			\varphi(\Delta)=\left[(1-t)^k+t^k\right]\varphi(\Delta)
		\end{align*}
		for every $t\in(0,1)$, which is only possible if $\varphi(\Delta)=0$ since $k\ne 1$. Thus $\varphi$ vanishes on all orthogonal simplices. Since $\varphi$ is simple, $\varphi$ thus has to vanish identically by Corollary \ref{corollary:SimpleVanishOrthogonalSimplices}, which completes the induction.
	\end{proof}
	
	It remains to combine the previous results with Hadwiger's characterization of $n$-homogeneous valuations in \autoref{theorem:HadwigerTopDegree}.
	\begin{proof}[Proof of \autoref{maintheorem:VolumeCharacterization}]
		If $\varphi=\sum_{k=0}^n\varphi_k$ is the decomposition of $\varphi$ into its homogeneous components, then $\varphi_k$ is a measurable, translation and $\SO(n)$-invariant, as well as simple valuation of degree $k$, and thus vanishes for $0\le k\le n-1$ by \autoref{proposition:simpleValuationsVanish}. Thus $\varphi=\varphi_n$ is $n$-homogeneous and therefore a multiple of the volume by \autoref{theorem:HadwigerTopDegree}. Note that this step uses that measurable translation invariant valuations are weakly continuous by \autoref{corollary:measurableImpliesWeaklyCont}.
	\end{proof}
	
	\begin{proof}[Proof of \autoref{maintheorem:HadwigerPolytopes}]
		We establish the result by induction on the dimension $n$, the case $n=1$ being trivial due to \autoref{theorem:HadwigerTopDegree}. Assume that the claim holds up to dimension $n-1$. Let $\varphi:\calP^n\rightarrow\R$ be a measurable, translation and $\SO(n)$-invariant valuation. We set $H:=e_n^\perp$ and consider the restriction of $\varphi$ to $\calP(H)$. By construction, this defines a measurable and translation invariant valuation on $\calP(H)$ that is invariant under the subgroup of $\SO(n)$ fixing $e_n$, which is isomorphic to $\SO(n-1)$. The induction assumption implies that there exist constants $c_0,\dots,c_{n-1}\in\R$ such that
		\begin{align*}
			\varphi(P_H)=\sum_{j=0}^{n-1}c_j V_j(P_H)
		\end{align*}
		for every $P_H\in\calP(H)$. The valuation $\tilde{\varphi}:=\varphi-\sum_{j=0}^{n-1} c_j V_j$ is therefore $\SO(n)$-invariant and vanishes on polytopes in $H$ (as well as measurable and translation invariant). Thus $\tilde{\varphi}$ is a simple valuation. By \autoref{maintheorem:VolumeCharacterization} there exists $c_n\in\R$ such that $\tilde{\varphi}=c_n V_n$, and we obtain
		\begin{align*}
			\varphi=\tilde{\varphi}+\sum_{j=0}^{n-1} c_jV_j=\sum_{j=0}^nc_j V_j.
		\end{align*}
	\end{proof}

	\section{Characterization of the Steiner point map}
	\label{section:SteinerPoint}
	We use a similar approach to obtain the characterization of the Steiner point map in \autoref{maintheorem:SteinerPoint}. The proof relies again on the characterization of simple $1$-homogeneous valuations in \autoref{theorem:Equiv1homValuation} for the standard representation of $\SO(n)$ on $\C^n$. We start by considering simple and translation invariant valuations. For simplicity, we identify $\mathcal{U}^{n-1}_n\cong \SO(n)$ using \autoref{lemma:isomorphismStiefelSOn}.
	\begin{lemma}\label{lemma:trivialitySimpleVectorValued}
		Let $n\ge 3$ and $\varphi:\calP^n\rightarrow\C^n$ be a measurable, $\SO(n)$-equivariant, translation invariant, and $1$-homogeneous valuation. If $\varphi$ is simple, then $\varphi=0$. 
	\end{lemma}
	\begin{proof}
		By \autoref{theorem:Equiv1homValuation}, there exists an odd function $f\in C(\SO(n),\C^n)^{\SO(n)}$ such that
		\begin{align*}
			\varphi(P)=\sum_{u\in \calU^{n-1}_n(P)}f(u)\vol_1(F(P,u)).
		\end{align*}
		Since $\C^n$ is an irreducible representation of $\SO(n)$ for $n\ge 3$, the Canonical Decomposition in \autoref{corollary:canonicalDecompTensor} implies
		\begin{align*}
			f\in C(\SO(n),\C^n)^{\SO(n)}\cong [(E\otimes E^*)\otimes \C^n]^{\SO(n)},
		\end{align*}
		where $\SO(n)$ operates trivially on the second factor on the right, $E\cong \C^n$, and $E\otimes E^*$ is spanned by the coordinate functions $x_{ij}:\SO(n)\rightarrow\R$. In other words, $f$ is a linear combination of $x_{ij}\otimes v$ for $1\le i,j\le n$ and $v\in \C^n$. However, since $n\ge 3$, no nontrivial linear combination of these elements defines an odd function. Thus $f=0$ and therefore $\varphi$ vanishes identically.
	\end{proof}

	\begin{lemma}\label{lemma:Steiner2d}
		Let $\varphi:\calP^2\rightarrow \C^2$ be an $\SO(2)$-equivariant, translation invariant and measurable valuation homogeneous of degree $1$. Then $\varphi=0$ 
	\end{lemma}
	\begin{proof}
		We claim that $\varphi$ is simple. Since it is $1$-homogeneous, $\varphi$ vanishes on points. If $I\in\calP^2$ is an interval, then $-I$ is a translate of $I$. Consequently, since $-Id\in \SO(2)$,
		\begin{align*}
			\varphi(I)=\varphi(-I)=-\varphi(I),
		\end{align*}
		which implies $\varphi(I)=0$. Thus $\varphi$ is simple and therefore satisfies the assumptions of \autoref{theorem:Equiv1homValuation}. We obtain an odd function $f\in C(\SO(2),\C^2)^{\SO(2)}$ such that
		\begin{align}
			\label{eq:1dimCaseSteiner}
			\varphi(P)=\sum_{u\in \calU^1_2(P)} f(u)\vol_1(F(P,u))=\int_{S^1} f(u)dS_1(P,u),
		\end{align}
		where we identify $\SO(2)\cong S^1$ using the first column, and where $S_1(P)$ denotes the surface area measure of $P$, compare \cite{SchneiderConvexbodiesBrunn2014}*{Chapter~4}. From the Canonical Decomposition, we obtain that 
		\begin{align*}
			C(\SO(2),\C^2)^{\SO(2)}\cong [\mathcal{H}_1\otimes \C^2]^{\SO(2)}
		\end{align*}
		where $\mathcal{H}_1$ denotes the space of harmonic polynomials of degree $1$ on $\SO(2)\cong S^1$. In other words, $f$ is contained in the space spanned by $x_1\otimes v$ and $x_2\otimes v$, $v\in \C^2$, where $x_1,x_2:\SO(2)\cong S^1\rightarrow\R$ denote the coordinate functions corresponding to the first column. Since the surface area measure $S_1(P)$ is centered, the integral in \eqref{eq:1dimCaseSteiner} vanishes, so $\varphi$ vanishes identically.
		
	\end{proof}

	\begin{proposition}\label{proposition:trivialityVectorValued}
		Let $n\ge 2$ and $\varphi:\calP^n\rightarrow \R^n$ be an $\SO(n)$-equivariant, translation invariant, and measurable valuation homogeneous of degree $1$. Then $\varphi=0$.
	\end{proposition}
	\begin{proof}
		We will argue by induction on $n\ge 2$, where the case $n=2$ is covered by \autoref{lemma:Steiner2d}. Let $n\ge 3$ and assume that the claim holds in all dimensions up to $n-1$. Assume that $\varphi$ is a valuation with these properties. We will show that $\varphi$ is simple.
		
		Let $H\subset \R^n$ be a linear subspace of dimension $2\le \dim H\le n-1$. If we restrict $\varphi$ to $\calP(H)$ and project orthogonally onto $H$, we obtain a map $\tilde{\varphi}:\calP(H)\rightarrow H$ that is an $\SO(H)$-equivariant, translation invariant and measurable valuation homogeneous of degree $1$. The induction assumption thus implies $\tilde{\varphi}=0$. Thus $\varphi(P)\in H^\perp$ for every $P\in \calP(H)$.
		
		Since $\SO(H)\subset \SO(n)$ acts trivially on $H^\perp$, $\varphi|_H:\calP(H)\rightarrow H^\perp$ is an $\SO(H)$ invariant valuation. By \autoref{maintheorem:HadwigerPolytopes}, the components of $\varphi|_H$ (with respect to any basis of $H^\perp$) are thus multiples of the mean width $V_1$ and therefore in particular $\mathrm{O}(H)$-invariant. Thus $\varphi|_H$ is $\mathrm{O}(H)$-invariant. If we choose $g\in \mathrm{O}(H)$ such that $\tilde{g}:=\begin{pmatrix}
			g & 0\\
			0 & -Id_{H^\perp}
		\end{pmatrix}\in \SO(n)$, then we obtain for every $P\in\calP(H)$ the relation
		\begin{align*}
			\varphi(P)=\varphi(gP)=\varphi(\tilde{g}P)=\tilde{g}\varphi(P)=-\varphi(P).
		\end{align*}
		Thus $\varphi(P)=0$, and since $\varphi$ is translation invariant, this shows that $\varphi$ is a simple valuation. Now \autoref{lemma:trivialitySimpleVectorValued} shows that $\varphi$ vanishes identically.
	\end{proof}

	\begin{proof}[Proof of \autoref{maintheorem:SteinerPoint}]
		Let $\varphi:\calP^n\rightarrow \R^n$ be measurable, translation and $\SO(n)$-equivariant, and Minkowski additive. If $s:\calP^n\rightarrow\R^n$ denotes the Steiner point map, then 
		\begin{align*}
			\tilde{\varphi}(P):=\varphi(P)-s(P)
		\end{align*}
		is an $\SO(n)$-equivariant, translation invariant, and measurable valuation that is homogeneous of degree $1$. By \autoref{proposition:trivialityVectorValued}, $\tilde{\varphi}$ vanishes identically, which shows $\varphi=s$.
	\end{proof}

\bibliographystyle{plain}
\bibliography{../../../library/library.bib}

\Addresses
\end{document}